\documentclass[11pt,reqno]{amsart}
\usepackage{setspace}
\usepackage[hang,flushmargin,symbol*]{footmisc}
\usepackage{amsmath}
\usepackage{amsthm}
\usepackage{amssymb}
\usepackage{mathtools}
\usepackage{enumitem}
\usepackage{calc}
\usepackage{graphicx}
\usepackage{caption}
\usepackage[labelformat=simple,labelfont={}]{subcaption}
\usepackage{tikz}
\usetikzlibrary{decorations.markings}
\usetikzlibrary{arrows,shapes,positioning}
\usetikzlibrary{patterns}
\usepackage{url}
\usepackage{array}
\usepackage{graphicx}
\usepackage{color}
\usepackage{mathrsfs}

\usepackage{verbatim}

\usepackage{subfiles}

\usepackage[colorinlistoftodos]{todonotes}

\usepackage{dashbox}

\theoremstyle{definition}
\newtheorem{theorem}{Theorem}[section]
\newtheorem{lemma}[theorem]{Lemma}

\newtheorem{corollary}[theorem]{Corollary}
\newtheorem{proposition}[theorem]{Proposition}

\newtheorem{definition}[theorem]{Definition}
\newtheorem{example}[theorem]{Example}
\newtheorem{remark}[theorem]{Remark}
\newtheorem{question}[theorem]{Question}

\newtheorem{situation}[theorem]{Situation}

\newtheorem{warning}[theorem]{Warning}

\usepackage{amsfonts}
\usepackage{amsmath}
\usepackage{amssymb}
\usepackage{mathrsfs}
\usepackage{tikz}
\usepackage{tikz-cd}
\usepackage{hyperref}
\usepackage{ mathdots }
\usepackage{dutchcal}
\usepackage{verbatim}
\usepackage{amsmath,amsthm}
\usepackage{extarrows}

\usepackage{xparse}

\usepackage{showlabels}

\usepackage{multirow}
\usepackage{multicol}

\usepackage[normalem]{ulem}
\newcommand{\scr}[1]{\ensuremath{\mathscr{#1}}}

\newcommand{\msout}[1]{\text{\sout{\ensuremath{#1}}}}

\newcommand{\ZZ}{\mathbb{Z}}
\newcommand{\CC}{\mathbb{C}}
\newcommand{\NN}{{\mathbb{N}}}

\newcommand{\OO}{\mathcal{O}}

\newcommand{\Aff}{{\mathbb{A}}}
\newcommand{\PP}{\mathbb{P}}

\newcommand{\GG}{\mathbb{G}}

\newcommand{\Gtrop}{\GG_{trop}}
\newcommand{\Glog}{\GG_{log}}

\newcommand{\Spec}{{\rm{Spec}\:}}
\newcommand{\SSpec}[1]{{\underline{\rm{Spec}}_{#1}\:}}

\newcommand{\Hom}{{\rm Hom}}
\newcommand{\HHom}{\underline{{\rm Hom}}}

\newcommand{\Der}{\text{Der}}

\newcommand{\ccx}[1]{{\mathbb{L}_{#1}}}

\newcommand{\lccx}[1]{{\mathbb{L}^\ell_{#1}}}

\newcommand{\gp}[1]{#1^{\rm gp}}

\newcommand{\ev}{\hat \wedge}

\newcommand{\Tl}[1]{{T^{\ell}_{#1}}}
\newcommand{\lkah}[1]{\Omega^\ell_{#1}}
\newcommand{\kah}[1]{\Omega_{#1}}
\newcommand{\Log}{{\mathcal{L}og}}

\newcommand{\lpb}{{\arrow[dr, phantom, very near start, "\ulcorner \ell"]}}
\newcommand{\lpbstrict}{{\arrow[dr, phantom, very near start, "\ulcorner \msout{\ell}"]}}
\newcommand{\pb}{{\arrow[dr, phantom, very near start, "\ulcorner"]}}

\newcommand{\num}[1]{\langle #1 \rangle}

\renewcommand{\tilde}[1]{\widetilde{#1}}
\renewcommand{\hat}[1]{\widehat{#1}}

\newcommand{\stquot}[1]{\left[ #1 \right]}

\def\overnorm#1{\overline{#1}\vphantom{#1}}

\renewcommand{\bar}[1]{\ensuremath{\overnorm{#1}}}

\newcommand{\Ms}{{\overline{M}_{g, n}}}

\newcommand{\Ml}{{\mathscr{M}^\ell_{g, n}}}

\newcommand{\Aut}{\underline{\text{Aut}}}

\newcommand{\Sym}{\text{Sym}\,}

\newcommand{\longequals}{\xlongequal{\: \:}}

\newcommand{\colim}{\text{colim}}

\usepackage{stmaryrd}
\newcommand{\adj}[1]{\llbracket #1 \rrbracket}

\newcommand{\ljet}[2]{\Delta_{#2}^{(#1)}}
\newcommand{\tljet}[2]{\scr D_{#2}^{(#1)}}
\newcommand{\jet}[1]{\Delta_{#1}}

\newcommand{\ljsp}[2]{J^{(#1)}_{#2}}
\newcommand{\tljsp}[2]{\scr J^{(#1)}_{#2}}
\newcommand{\ljspnonfs}[2]{S^{(#1)}_{#2}}
\newcommand{\jsp}[1]{J_{#1}}
\newcommand{\tjsp}[1]{\scr J_{#1}}

\newcommand{\WR}{\cal R}

\newcommand{\Gen}{\cal M}

\newcommand{\smet}{\Acute{e}t}

\newcommand{\bra}[1]{\left[{#1}\right]}

\usepackage{bm}
\usepackage{xspace}
\newcommand{\AF}[1][{}]{\Theta_{{#1}}}
\newcommand{\af}[1][{}]{\AF[#1]}

\newcommand{\cal}[1]{\ensuremath{\mathcal{#1}}}

\newcommand{\pt}{\Spec k}

\newcommand{\VV}{\mathbb{V}}

\newcommand{\rk}{{\rm rk}}

\usepackage{subfiles}

\newcommand{\notinsubfile}[1]{}

\newcommand{\ord}{{\rm ord}}
\newcommand{\Fitt}{{\rm Fitt}}

\newcommand{\embdim}{{\rm emb.dim}}

\newcommand{\longsimeq}{\overset{\sim}{\longrightarrow}}

\title[Log jet spaces]{The log tangent space of the log jet space}
\author[Leo Herr]{Leo Herr$^1$}
\email{l.s.herr@math.leidenuniv.nl}
\date{\today \qquad
$\phantom{a}^1$Universiteit Leiden}

\begin{document}

\maketitle

\begin{abstract}
We introduce new notions of log jet spaces. Mildly singular spaces are ``smooth'' in log geometry, so their log jet spaces behave like the jet spaces of smooth varieties. Myriad examples contrast log jet spaces with the usual jet spaces of schemes. 

We then compute the log Kähler differentials of the log jet and arc spaces after de Fernex and Docampo. We obtain some of their applications to the structure of the log jet space. We conclude with comparison remarks with other log jet spaces. 
\end{abstract}

\section{Introduction}

Varieties with locally toroidal singularities can be made log smooth with suitable log structure. These include divisors with normal crossings, nodal curves, and all toric varieties. For these singular varieties, we introduce log jet spaces that behave like the jet spaces of smooth varieties. For varieties without log structure, log jet spaces are the usual jet spaces. We assume the reader is familiar with log structures at the level of \cite{katooriginal} or \cite{ogusloggeom}.

The log jet spaces $\ljsp{r}{X, m}$ of a target variety $X$ depend on the order of the jets $m$ and a new logarithmic parameter $r \in \NN \cup \{\infty\}$. When $m = \infty$, $\ljsp{r}{X, \infty}$ is called the \emph{log arc space}. 

Our long-term goals concern log motivic integration. The present article has three goals: 
\begin{enumerate}
    \item Relate the log jet space $\ljsp{r}{X, m}$ to the evaluation stack $\wedge X$ in log Gromov--Witten theory \cite{evaluationspacefamiliesofstandardlogpoints}. 
    \item Compute the log Kähler differentials $\lkah{\ljsp{r}{X, m}}$ of the log jet space $\ljsp{r}{X, m}$ in terms of those of $X$. 
    \item Obtain numerical applications to the embedding dimension of the log jet space at a point $\alpha \in \ljsp{r}{X, m}$. 
\end{enumerate}
The relation to $\wedge X$ establishes representability of $\ljsp{r}{X, m}$ by a log algebraic space. It also gives a new evaluation map from the space of stable maps $C \to X$ from curves to a variant of the jet space. 

The formula for the log Kähler differentials of the log jet space answers a question of L. Ein in log geometry. Knowing the tangent space of the arc space gives numerical formulas that dictate its structure \cite[Theorems D, E]{arcspacedifflsdefernexdocampo}. These numerical invariants are more subtle in log geometry.

\subsection{Representability}

We first show log jet spaces of fs log schemes $X$ are representable, mimicking \cite{evaluationspacefamiliesofstandardlogpoints}. 

\begin{theorem}[Theorem \ref{thm:repability}]

Let $X$ be an fs log scheme. The log jet spaces $\ljsp{r}{X, m}$ are representable by fs log algebraic spaces. If $X$ admits charts Zariski locally, $\ljsp{r}{X, m}$ is representable by an fs log scheme. 

\end{theorem}

A big difference from ordinary jets is that the order $m = 0$ jet space is not simply the target $\ljsp{r}{X, 0} \neq X$. The $m=0$ log jet $\ljet{r}{0}$ is not $\Spec k$. The \textit{evaluation space} $\ev X \coloneqq \ljsp{r}{X, 0}$ is the case $m = 0$.

By \cite{loghomstackwise}, $\ev X$ is representable by a log algebraic space. Section \ref{ss:moreevsp} includes tons of examples of evaluation spaces. This space is similar to its namesake \emph{evaluation stack} $\wedge X$ of \cite{evaluationspacefamiliesofstandardlogpoints}. 

\begin{proposition}[Proposition \ref{rmk:compareevalspabram}]
There is a strict quotient map $\ev X \to \wedge X$ by an action of $\GG_m$.  
\end{proposition}

The purpose of the evaluation stack is to receive a map $ev : \Ml(X) \to (\wedge X)^n$. Here $\Ml(X)$ is the moduli space of log stable maps $C \to X$ from families of curves of genus $g$ with $n$ marked points. The evaluation map $ev$ restricts a curve $C \to X$ to its $n$ marked points. 

Let $\GG_m$ act on the log arc space $\ljsp{r}{X, \infty}$ by rescaling the arcs. We show the evaluation map $ev$ factors through the quotient $\stquot{\ljsp{r}{X, \infty}/\GG_m}$ of the log arc space by this action. The factorization restricts a curve $C \to X$ to the infinitesimal neighborhood of its marked points. 

\begin{theorem}[Theorem \ref{thm:jetspevmap}]
The evaluation map $ev$ on log stable maps factors through the quotient of the log jet space
\[ev : \Ml(X) \to \left(\stquot{\ljsp{r}{X, m}/\GG_m}\right)^n \to (\wedge X)^n.\]
\end{theorem}

We also obtain a version of B. Bhatt's result \cite{bhatttannakaalgn} that the arc space $\jsp{X, \infty}$ is the limit of the jet spaces $\lim \jsp{X, m}$. This question was open for a long time, as there is no easy way to turn a system of $R \adj{t}_m$-points into a single point over $R \adj{t}$. The proof uses infinity categories and exchanges maps of varieties $X \to Y$ for the associated symmetric monoidal pullback functors $D_{perf}(Y) \to D_{perf}(X)$ on perfect objects in their derived categories. 

\begin{theorem}[Proposition \ref{prop:bhattthm}]

Let $X$ be a log scheme. The log arc space is the limit of the corresponding log jet spaces
\[\ljsp{r}{X, \infty} = \lim_m \ljsp{r}{X, m}.\]

\end{theorem}

\subsection{The log tangent space of the log jet space}

We return to a question L. Ein posed to T. de Fernex. 

\begin{question}[L. Ein]\label{q:difflsarcsp}
What are the Kähler differentials of the arc space $\jsp{X, \infty}$ of a variety $X$?
\end{question}

Theorem B of \cite{arcspacedifflsdefernexdocampo} answers this question. We revisit this question for other Weil restrictions before answering it for log jet/arc spaces. This has numerical consequences (Theorem D, E loc. cit.) for the structure of the jet space. 

Write $\cal U$ for the universal log jet $\cal U \coloneqq \ljsp{r}{X, m} \times \ljet{r}{m}$, with $r, m$ understood. We answer Question \ref{q:difflsarcsp} in log geometry.

\begin{theorem}[{\cite[Theorem B]{arcspacedifflsdefernexdocampo}, Corollary \ref{cor:jetspdiffls}}]
The log K\"ahler differentials of the jet space are:
\[\lkah{\ljsp{r}{X, m}} = \rho_*\left(\lkah{X}|_{\cal U} \otimes \omega\right),\]
where $\omega$ is the dualizing sheaf of $\jet{m} \to \pt$, pulled back to $\cal U$. 

\end{theorem}

We give a new geometric proof of \cite{arcspacedifflsdefernexdocampo} similar to that of \cite{hasseschmidtmodulechiu} and generalize the result. The proof shows that the log tangent space of the log jet space is the Weil restriction
\[\Tl{\ljsp{r}{X, m}} \simeq \WR_{\cal U/\ljsp{r}{X, m}} \left(\Tl{X}|_{\cal U}\right)\]
of the log tangent space $\Tl{X}$ of the target along the universal jet $\cal U \to \ljsp{r}{X, m}$. 

We can describe the tangent space of general Weil restrictions, including jet spaces, configuration spaces, and the monogeneity spaces of \cite{mymonogeneity1}. 

The formula for the log tangent space of the jet space lets one compute the embedding dimension (Definition \ref{defn:embdim}) of the log jet space at a point $\alpha_m \in \ljsp{r}{X, m}$.

\begin{theorem}[{\cite[Lemma 8.1]{arcspacedifflsdefernexdocampo}, Theorem \ref{thm:embdim}}]

Suppose $X$ is a finite type, log smooth log scheme and the ground field $k$ is perfect. Let $\alpha_m \in \ljsp{r}{X, m}$ be a point which lifts to an arc $\alpha_\infty \in \ljsp{r}{X, \infty}$. The embedding dimension $\embdim\left(\OO_{\ljsp{r}{X, m}, \alpha_m}\right)$ of the log jet space at $\alpha_m$ is given by
\[
\begin{split}
\embdim\left(\OO_{\ljsp{r}{X, m}, \alpha_m}\right) = &d_m(m+1) + \ord_{\alpha_\infty} \left(\Fitt^{d_m}(\lkah{X}|_{\alpha_\infty \times \ljet{r}{\infty}})\right) \\
&- \rk \gp{\bar M}_{\alpha_m} + N - \dim_k \alpha^\circ.
\end{split}
\]
The number $N$ is the number of irreducible elements in the stalk of the sheaf $\bar M_\alpha$ and $d_m$ is the Betti number of $\lkah{X}$ at $\alpha_m$. 

\end{theorem}

This formula gives an inequality if $X$ is not log smooth. The embedding dimension is familiar to Gromov--Witten theorists as the dimension of a certain \textit{intrinsic normal sheaf} (after Definition \ref{defn:embdim}).

The appendix checks that the jet space of an algebraic stack is an algebraic stack, an enjoyable exercise in the definitions. Log algebraic stacks then have representable log jet spaces as well (Corollary \ref{cor:ljetrepablestacks}).

The new log structure can describe the space of arcs of fixed contact order along a divisor ${\rm Cont}^{\geq m}(D)$ used in motivic integration. Recent approaches to motivic integration using Berkovich spaces \cite{tommasodefernexberkovichjets} ignore jets entirely and work only with valuations. We hope to use these spaces for motivic integration and compare with \cite{tommasodefernexberkovichjets} in future work.

\subsection{Conventions}

We consider schemes as a full subcategory of log schemes, endowing a scheme $S$ with the trivial log structure 
\[M_S = \OO_S^*.\]
Log schemes $X$ have underlying scheme $X^\circ$, which is automatically given its trivial log structure $M_{X^\circ} = \OO_X^*$. We use adjectives like ``ordinary'' to refer to the case of schemes, as opposed to log schemes. 

Pullbacks $X \times_Z Y$ in the category of fs log schemes differ from the pullback of the underlying schemes. We write $\times^\ell$ for the fs pullback, $\times$ for that in schemes, and $\times^{\msout{\ell}}$ if they happen to coincide. Cartesian diagrams similarly have $\ulcorner, \ulcorner \ell,$ or $\ulcorner \msout{\ell}$ if they are of ordinary schemes, fs log schemes, or both. If the distinction is important, we emphasize it -- these symbols merely indicate to the reader which is the relevant category.

We work over a field $k$ of arbitrary characteristic, not necessarily algebraically or separably closed. We require $k$ to be perfect in Section \S \ref{s:fittingideals}, but not in the rest of the paper. 

If $P$ is a monoid, write
\[\Aff_P \coloneqq \Spec k[P].\]
The Cartier dual $\Aff_{\gp P}$ is the dense torus. The stack quotient 
\[\af[P] \coloneqq \bra{\Aff_P/\Aff_{\gp P}}\]
is an Artin cone, and stacks with \'etale covers by such are Artin fans. We write $P^+ = P \setminus P^*$ for the nonunits of $P$; e.g. $\NN^+ = \NN \setminus 0$. 

We write $\smet(S)$ for the small \'etale site of a scheme. If $Q$ is a set, write $\underline{Q}$ for the constant sheaf on the small \'etale site $\smet(S)$ of some scheme $S$.

If $F$ is a coherent sheaf on a scheme $X$, write 
\[\VV(F) \coloneqq \SSpec{X} \Sym F\]
for the associated scheme affine over $X$. If $F$ is locally free, $\VV(F)$ is the associated vector bundle.   

Write $\lccx{X/Y}$ for the log cotangent complex \cite{logcotangent} of a morphism $X \to Y$ and $h^{-1} \lccx{X/Y}$ for its cohomology group in degree $-1$. 

We write $\Log$ for Olsson's stack of fs log structures (``$\cal T or$'' in the original).

\subsection{Acknowledgments}

This paper could not have been written without Tommaso de Fernex, who gracefully declined to be a coauthor. The ideas and results were nevertheless obtained in collaboration with him throughout. Gebhard Martin helped with the original definitions and earliest examples. 

Jonathan Wise fixed a mistake in the essential Example \ref{ex:evspA1}, without which the paper falls apart. Section \ref{s:logtansplogjetsp} also arose from conversations with Wise, where he patiently explained the notion of an obstruction theory to the author. 

Thanks also go to Sarah Arpin, Andreas Gross, Sebastian Casalaina-Martin, Sam Molcho, and Johannes Nicaise. 

David Holmes tremendously improved the manuscript, especially the introduction. The author is grateful for his careful reading and rereading through versions of the present article. His comments also inspired the relation to log Gromov--Witten theory in Section \S \ref{ss:logjetloggw}. 

The author received partial support from NSF RTG
grant \#1840190 during the writing of this paper.

\section{Log jet spaces}

This section shows that $\ljsp{r}{X, m}$ is representable by a log algebraic space, or log scheme if $X$ admits charts Zariski locally. The proof exploits the relationship of $\ljsp{r}{X, m}$ with the evaluation stack $\wedge X$ in log Gromov--Witten theory.

Work over an arbitrary base field $k$ and give $\Spec k$ trivial log structure. Write
\[\begin{tikzcd}
k \adj{t}_m \coloneqq k[t]/t^{m+1},         &\text{for }m \neq \infty      \\
k \adj{t}_\infty \coloneqq k \adj{t}        &\text{for }m = \infty   \\
\jet{m} \coloneqq \Spec k \adj{t}_m.
\end{tikzcd}\]
This notation lets us write jets and arcs uniformly. Ordinary jet spaces $\jsp{X, m}$ represent maps from these jets to the scheme $X$:
\begin{equation}\label{eqn:ordjetfunctoropoints}
\jsp{X, m}(\Spec A) \coloneqq \Hom(\Spec A \times \jet{m}, X).   
\end{equation}
If $m = \infty$, one takes the fiber product $\Spec A \times \jet{\infty}$ as formal schemes to obtain $\Spec A \adj{t}$. We endow the jets $\jet{m}$ with different log structures to define divers log jet spaces. 

\begin{definition}

Let $r > 0$ be a positive integer. Define a log scheme $\ljet{r}{m}$ with underlying scheme $\jet{m}$ and the log structure associated to
\[\NN \to k \adj{t}_m \qquad 1 \mapsto t^r.\]
Extend this definition to $r = \infty$ by taking the map 
\[\NN \to k \adj{t}_m \qquad 1 \mapsto ``{t^\infty}" = 0.\]

\end{definition}

The case $r = 1$ records the $t$-adic valuation of the jet $\jet{m}$.

\begin{definition}

Let $X$ be an fs log scheme, $r > 0$ an integer, and $T$ an fs log scheme over $k$. Define the {\em log jet space} $\ljsp{r}{X, m}$ by
\[\ljsp{r}{X, m}(T) \coloneqq \Hom(T \times \ljet{r}{m}, X). \]
The fiber product takes place in the category of fs log schemes, or formal fs log schemes if $m = \infty$. The \textit{log arc space} refers to the case $m = \infty$. 

\end{definition}

The log jet space is functorial in $X$. If $T = \Spec A$ is affine, the underlying scheme of $T \times \ljet{r}{m}$ is $\Spec A \adj{t}_m$\footnote{
The schematic and fs fiber products coincide 
\[(T \times^\ell \ljet{r}{m})^\circ = T^\circ \times \jet{m}\]
because $\ljet{r}{m} \to \pt$ is integral and saturated.
}.

\begin{remark}

S. Dutter \cite{sethdutterthesislogjets} and B. Fleming \cite{balinthesis} used log jet spaces that correspond to taking $r = 0$. We defer the case $r = 0$ to Section \ref{s:r=0} and do not allow $r = 0$ until then. 

Their version has the advantage that the first jet space is the log tangent space $\ljsp{0}{X, 1} = \Tl{X}$ and the zeroth $\ljsp{0}{X, 0} = X$ coincides with $X$. Their jets have trivial log structure however, so the $t$-adic valuation encoding contact order information is not visible in their log structure. 

\end{remark}

\begin{definition}\label{defn:evalsp}

Let $X$ be an fs log scheme and write $P_{\NN} = \ljet{r}{0}$ for the log $0$-jet, the point $\Spec k$ with rank-one log structure. The \textit{evaluation space} is the log jet space with $m = 0$
\[\ev X \coloneqq \ljsp{r}{X, 0}.\]
Its functor of points on fs log schemes $T$ is
\[\ev X(T) \coloneqq \Hom(T \times P_{\NN}, X).\]

\end{definition}

Our first examples are the affine spaces $X = \Aff^1, \Aff^n$. To describe $\ljsp{r}{X, m}$ for $X = \Aff^1$, we need to understand the truncation maps and the ``base case'' $\ev X = \ljsp{r}{X, 0}$ with $m = 0$.

\begin{remark}[Truncation maps]

For any $r$, there are strict closed immersions
\[\ljet{r}{0} \subseteq \ljet{r}{1} \subseteq \cdots \subseteq \ljet{r}{\infty}. \]
Restricting along these maps yields truncation maps between the log jet spaces:
\[\ljsp{r}{X, \infty} \to \cdots \to \ljsp{r}{X, m+1} \to \ljsp{r}{X, m} \to \cdots \to \ljsp{r}{X, 0}.\]
Each map truncates an $(m+1)$-jet to an $m$-jet. 

\end{remark}

The log structure on the log jets is
\[\Gamma(M_{\ljet{r}{m}}) = \NN \times \Gamma((k \adj{t}_m)^*).\]
The map $M_{\ljet{r}{m}} \to \OO_{\ljet{r}{m}}$ sends $(n, u)$ to $u \cdot t^n \in k \adj{t}_m$. The monoid $M_{\ljet{r}{m}}$ is isomorphic to $k \adj{x}_m \setminus 0$, and the map sends $x$ to $t^r$. These log structures are discussed for discrete valuation rings in \cite[\S III.1.5]{ogusloggeom} and appear in \cite[Example 6.13]{logcotangent}. 

Whenever $r > m$, $t^r = 0$ and all the log jets coincide:
\[\ljet{m+1}{m} = \ljet{m+2}{m} = \cdots = \ljet{\infty}{m}.\]

\begin{remark}[Trivial log structure]

If $X = X^\circ$ has trivial log structure, the log jet space coincides with the ordinary jet space: 
\[\ljsp{r}{X, m} = \jsp{X, m}.\]
Even if $X$ does not have trivial log structure, there is a natural map $X \to X^\circ$ inducing a map from the log jet space to the ordinary one: 
\[\ljsp{r}{X, m} \to \jsp{X, m}. \]
This sends log jets to the underlying diagram of ordinary jets. 

\end{remark}

\begin{warning}

There is no map $\ljsp{r}{X, m} \dashrightarrow X$ of log schemes. The truncation to order $m=0$ gives a map to the evaluation space $\ljsp{r}{X, m} \to \ev X$ instead. 

If $Z = Z^\circ$ has trivial log structure, we know $\ev Z = Z$. The map $X \to X^\circ$ then yields a map $\ev X \to X^\circ$ to the underlying scheme, but this is not compatible with log structures. 

This happens because the zero jet $P_{\NN}$ does not receive a map from the point $\Spec k \dashrightarrow P_{\NN}$. Such a map would entail a diagram 
\[\begin{tikzcd}
k^\ast \ar[d]         &\NN \oplus (k \adj{t}_m)^\ast \ar[d] \ar[l, dashed]         \\
k       &k \adj{t}_m, \ar[l]
\end{tikzcd}\]    
but the generator $1 \in \NN$ maps to $t^r$, which goes to $0 \in k$. Since $0 \in k$ is not in the image of $k^\ast \subseteq k$, there's no way to extend the ring map to one of log structures. 

There is a map $P_{\NN} \to P_\NN^\circ = \Spec k$ the other way. This gives a map $X \to \ev X$ which is always an inclusion of components. 

\end{warning}

Pointed log structures \cite{pointedlogstructures} and idealized log schemes \cite[\S III.1.3]{ogusloggeom} allow a map $\Spec k \dashrightarrow P_{\NN}$ by adding a preimage ``$\infty$'' of $0 \in k$ to the sheaf of monoids. Then closed immersions cut out by monoidal ideals can extend to immersions of log schemes with the natural log structures on each. 

One can think of the map $\ev X \to X^\circ$ as coming from the universal map $\ev X \times P_{\NN} \to X$. The source $\ev X \times P_{\NN}$ is not the same as $\ev X$, but it has the same underlying scheme. So the universal map $\ev X \times P_{\NN} \to X \to X^\circ$ to the underlying scheme of $X$ factors through $(\ev X \times P_{\NN})^\circ = (\ev X)^\circ$.

We will often use some basic properties of $\ljsp{r}{X, m}$ without mention. We introduce these before moving on to the case $X = \Aff^1$.

\begin{lemma}\label{lem:letpb}

If $X \to Y$ is log \'etale, the truncation maps fit in a cartesian diagram of fs log schemes 
\begin{equation}\label{eqn:wedgepbljets}
\begin{tikzcd}
\ljsp{r}{X, m+1} \ar[r] \ar[d] \lpb      &\ljsp{r}{Y, m+1} \ar[d]         \\
\ljsp{r}{X, m} \ar[r] \ar[d] \lpb      &\ljsp{r}{Y, m} \ar[d]         \\
\ev X \ar[r]        &\ev Y.
\end{tikzcd}    
\end{equation}
This applies to all $m$, including $m = \infty$. 

\end{lemma}

\begin{proof}

This is evident for $m$ finite. For $m = \infty$, use Bhatt's theorem below (Proposition \ref{prop:bhattthm}). 

\end{proof}

More is true: The map $\ljsp{r}{X, m} \to \ljsp{r}{Y, m}\times_{\ev Y}^\ell \ev X$ is surjective if $f$ is merely log smooth and injective if $f$ is log unramified.

\begin{remark}
If $X \to Y$ is strict, the square
\[\begin{tikzcd}
\ljsp{r}{X, m} \ar[r] \ar[d] \lpbstrict      &\ljsp{r}{Y, m} \ar[d]         \\
\jsp{X, m} \ar[r]      &\jsp{Y, m}
\end{tikzcd}\]
is cartesian in fs log schemes and underlying schemes. The jet spaces $\jsp{X, m}, \jsp{Y, m}$ here have trivial log structures. When $m=0$, $\jsp{X, m} = X^\circ$ and the square is even simpler
\[\begin{tikzcd}
\ev X \ar[r] \ar[d] \lpbstrict      &\ev Y \ar[d]         \\
X^\circ \ar[r]      &Y^\circ.
\end{tikzcd}\]

\end{remark}

\begin{remark}\label{rmk:logjetspacepb}

If a diagram
\[\begin{tikzcd}
X \ar[r] \ar[d] \lpb       &Y \ar[d]      \\
Z \ar[r]       &W,
\end{tikzcd}\]
is an fs pullback square, so is the diagram of log jet spaces:
\[\begin{tikzcd}
\ljsp{r}{X, m} \ar[r] \ar[d] \lpb       &\ljsp{r}{Y, m} \ar[d]      \\
\ljsp{r}{Z, m} \ar[r]       &\ljsp{r}{W, m}.
\end{tikzcd}\]
\end{remark}

\subsection{The case $X = \Aff^1$}

Recall the functor of points of $\Aff^1$ with its natural log structure.

\begin{example}[Maps into $\Aff^1$]

A map from any log scheme $T$ to $\Aff^1$ is a section $s \in \Gamma(M_T)$. The square 
\[\begin{tikzcd}
\Gamma(M_T) \ar[d]        &\NN \times (k[x])^* \ar[l] \ar[d]        \\
\Gamma(\OO_T)       &k[x] \ar[l]
\end{tikzcd}\]
is determined by where $1 \in \NN$ goes under the upper horizontal arrow. 

\end{example}

We need to understand the truncation maps for $X = \Aff^1$. Lifts of an $S = \Spec A$-valued $m$-jet to an $(m+1)$-jet are the dashed arrows
\[\begin{tikzcd}
&\ljsp{r}{\Aff^1, m+1} \ar[d]    &\phantom{a}\ar[d, phantom, "\leftrightsquigarrow"]       &S \times \ljet{r}{m} \ar[r] \ar[d]         &\Aff^1             \\
S \ar[r] \ar[ur, dashed]    &\ljsp{r}{\Aff^1, m}     &\phantom{a}       &S \times \ljet{r}{m+1} \ar[ur, dashed]
\end{tikzcd}.\]

There is a short exact sequence for the units of the closed immersion $S \times \ljet{r}{m} \subseteq S \times \ljet{r}{m+1}$:
\[1 \to 1 + A t^{m+1} \to A\adj{t}_{m+1} \to A \adj{t}_m \to 1.\]
The kernel is isomorphic to the additive group $1 + At^{m+1} \simeq (A, +)$ of the ring $A$. Similarly, the restriction map on log structures
\[M_{S \times \ljet{r}{m+1}} \to M_{S \times \ljet{r}{m}}\]
is a torsor for $1 + A t^{m+1}$. It is a torsor and not merely a pseudo-torsor because $\Aff^1$ is log smooth and lifts exist locally.

Lifting an $m$-jet to an $(m+1)$-jet requires lifting a section $s \in \Gamma(M_{S \times \ljet{r}{m}})$ along this restriction map. The truncation maps are therefore torsors for the additive group of $A$, represented by the group scheme $(\Aff^1)^\circ$.

This argument generalizes to $X = \Aff^n$.

\begin{lemma}\label{lem:logsmaffinebund}

Let $X = \Aff^n$. The restriction maps
\[\ljsp{r}{\Aff^n, m+1} \to \ljsp{r}{\Aff^n, m}\]
are strict torsors for $(\Aff^n)^\circ$. In particular, they are affine bundles. 

\end{lemma}

These affine bundles are rarely vector bundles in general. For $X = \Aff^n$, we will be able to trivialize them.

\begin{corollary}

Let $X$ be a log smooth log scheme of dimension $n$. The truncation maps 
\[\ljsp{r}{X, m+1} \to \ljsp{r}{X, m}\]
are strict affine bundles of relative dimension $n$. 

\end{corollary}

\begin{proof}

Because $X$ is log smooth, one can strict-\'etale locally find a log \'etale map $X \to \Aff^n$ \cite[Theorem IV.3.2.6]{ogusloggeom}. This pullback square concludes the argument:
\[\begin{tikzcd}
\ljsp{r}{X, m+1} \ar[r] \ar[d] \lpbstrict        &\ljsp{r}{\Aff^n, m+1} \ar[d]      \\
\ljsp{r}{X, m} \ar[r]     &\ljsp{r}{\Aff^n, m}.
\end{tikzcd}\]    

\end{proof}

\begin{remark}

We suspect the truncation maps $\ljsp{r}{X, m+1} \to \ljsp{r}{X, m}$ are always strict. We have a sketch of a proof using Wise's description of the minimal log structure \cite{loghomstackwise}, but this construction does not expressly work for jet spaces because jets are not reduced. 

\end{remark}

Once we describe the case $\ev \Aff^1$ with $m = 0$, the higher order jets $\ljsp{r}{\Aff^1, m}$ are just affine bundles over it.

\begin{example}[The evaluation space $\ev \Aff^1$]\label{ex:evspA1}

This example is an observation of Jonathan Wise. See Figure \ref{fig:evspA1}. 

A map $T \to \ev \Aff^1$ amounts to $T \times P_{\NN} \to \Aff^1$, or a section of the log structure of $T \times P_{\NN}$. Because $P_{\NN}$ has log structure $M_{P_\NN} = k^* \oplus \NN$, the log structure on $T \times P_{\NN}$ is
\[M_{T \times P_{\NN}} = M_T \oplus \underline{\NN},\]
where $\underline{\NN}$ is the constant sheaf. The product $\Gamma(M_T \oplus \underline{\NN}) = \Gamma(M_T) \oplus \Gamma(\underline{\NN})$ means that such a section is a pair of 
\begin{itemize}
    \item A section $\Gamma(M_T)$ is representable by $\Aff^1$,
    \item A section $\Gamma(\underline{\NN})$ is representable by the infinite disjoint union $\bigsqcup_\NN \Spec k$. 
\end{itemize}

The space $\ev \Aff^1$ is then representable by the product
\[\Aff^1 \times \bigsqcup_\NN \Spec k = \bigsqcup_\NN \Aff^1.\]
The universal map from $\ev \Aff^1 \times P_{\NN} = \bigsqcup_\NN \Aff^1 \times P_{\NN}$ to $\Aff^1$ parameterizes the sum of the universal element of $\Gamma(M_{\Aff^1})$ and the contact order on each component. The map to the underlying scheme $(\Aff^1)^\circ$ is thus projection to the origin for all of the copies of $\Aff^1$ except for the one corresponding to $0 \in \NN$. 

For the standard log point $P_{\NN} \in \Aff^1$ at the origin, the evaluation space $\ev P_{\NN}$ is pulled back from $\ev \Aff^1$. The $c=0$ component is $P_{\NN}$ as always, but the others are all $\Aff^1$. 

\end{example}

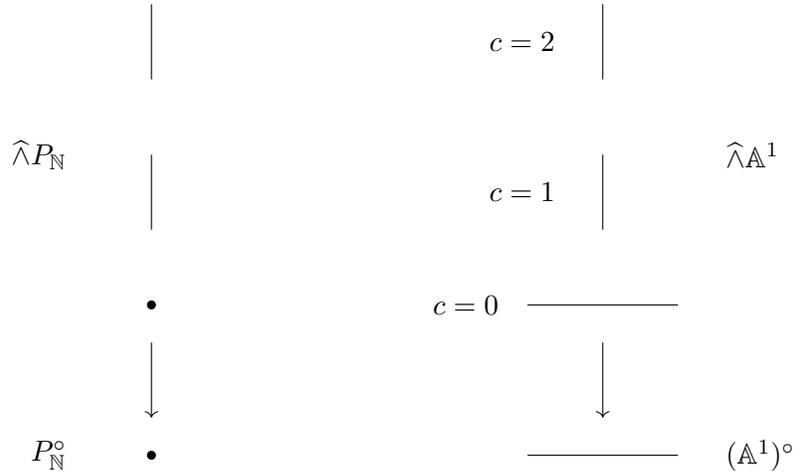
\begin{figure}
    \centering
    \begin{tikzpicture}
    \draw[fill] (0, -1) circle (1.5 pt);
    \draw[fill] (0, 1) circle (1.5 pt);
    \draw[-] (0, 2) to (0, 3);
    \draw[-] (0, 4) to (0, 5);
    \node[left] at (-1, -1){$P_\NN^\circ$};
    \node[left] at (-1, 3){$\ev P_{\NN}$};
    \draw[->] (0, .5) to (0, -.5);
    \begin{scope}[shift = {(6, 0)}]
    \draw[-] (-1, -1) to (1, -1);
    \draw[-] (-1, 1) to (1, 1);
    \node[left] at (-1.25, 1){$c=0$};
    \draw[-] (0, 2) to (0, 3);
    \node[left] at (-.5, 2.5){$c=1$};
    \draw[-] (0, 4) to (0, 5);
    \node[left] at (-.5, 4.5){$c=2$};
    \draw[->] (0, .5) to (0, -.5);
    \node[right] at (1.5, -1){$(\Aff^1)^\circ$};
    \node[right] at (1.5, 3){$\ev \Aff^1$};
    \end{scope}
    \end{tikzpicture}
    \caption{Right: The contact order $c=0$ component of $\ev \Aff^1$ maps identically to $\Aff^1$, while the other components $c = 1, 2$ map to the origin in $(\Aff^1)^\circ$. Restricting to the origin $P_\NN^\circ \in (\Aff^1)^\circ$ gives $\ev P_{\NN}$ on the left. The picture is similar for $\ev_Q \Aff^1$ with $Q$ not equal to $\NN$, but with more components indexed by all $c \in Q^+$. }
    \label{fig:evspA1}
\end{figure}

\begin{remark}\label{rmk:nonfsevspA1}

The evaluation space $\ev \Aff^1$ represents the functor 
\[T \mapsto \Hom(T \times P_{\NN}, \Aff^1)\]
on \textit{all} log schemes, not just saturated, integral, or even coherent ones. We doubt this is true for evaluation spaces $\ev X$ of other targets $X$ or the log jet spaces $\ljsp{r}{X, m}$. 

\end{remark}

\begin{corollary}

The evaluation space $\ev \Aff^n$ is a disjoint union of copies of $\Aff^n$ indexed  by $\NN^n$
\begin{equation}\label{eqn:evaffinenspace}
 \ev \Aff^n = \bigsqcup_{\NN^n} \Aff^n.   
\end{equation}
The log jet spaces are strict affine bundles over these, a disjoint union of affine spaces: 
\begin{equation}\label{eqn:logjetsAn}
\ljsp{r}{\Aff^n, m} = \bigsqcup_{\NN^n} \left(\Aff^{nm}\right)^\circ \times \Aff^n.    
\end{equation}

\end{corollary}

\begin{proof}

The equation \eqref{eqn:evaffinenspace} results from Example \ref{ex:evspA1} and compatibility of $\ev-$ with limits. 

The only surprise is that the affine bundles can be trivialized. Since the log jet space preserves fs fiber products $\ljsp{r}{\Aff^n, m} = (\ljsp{r}{\Aff^1, m})^n$, it suffices to trivialize the case of $\Aff^1$:
\[\ljsp{r}{\Aff^1, m} \overset{?}{=} \bigsqcup_\NN \left(\Aff^m \right)^\circ \times \Aff^1.\]
Induct on $m$. The order $m+1$ jets are an $(\Aff^1)^\circ$-bundle over $\ljsp{r}{\Aff^1, m} = \bigsqcup_\NN \left(\Aff^m\right)^\circ \times \Aff^1$. Such $(\Aff^1)^\circ$-bundles come from $\Aut(\Aff^1)$-torsors. Automorphisms of $(\Aff^1)^\circ$ are affine transformations: 
\[\Aut(\Aff^1) = \GG_m \ltimes \Aff^1.\]
But both $\GG_m$ and $\Aff^1$ torsors are trivial on each component $\left(\Aff^m\right)^\circ \times \Aff^1$. 

\end{proof}

\begin{remark}

We reduce to $X = \Aff^1$ in the proof. This proof would not work directly for the log jet space of $\Aff^n$, because $\Aff^n$-bundles can be much more twisted. The automorphism group of $\Aff^n$ is much larger than the affine transformations. Even $\Aut(\Aff^2)$ includes the map:
\[x \mapsto x, \qquad y \mapsto y + x^3.\]

\end{remark}

\begin{remark}

Ordinary Hasse-Schmidt differentials give coordinate functions on the jet spaces of $\Aff^n$ \cite[Corollary 5.2]{vojtajetspaces}. The coordinates on the components \eqref{eqn:logjetsAn} may be considered as ``log Hasse-Schmidt differentials.'' One may be able to describe a canonical trivialization using such higher log differentials directly. 

\end{remark}

\subsection{Representability}

We are ready to prove representability of the log jet space.

\begin{theorem}\label{thm:repability}

Let $X$ be an fs log scheme. The log jet spaces $\ljsp{r}{X, m}$ are representable by fs log algebraic spaces. If $X$ admits charts Zariski locally, $\ljsp{r}{X, m}$ is representable by an fs log scheme. 

\end{theorem}

One cannot use \cite{loghomstackwise} for $m \neq 0$, because his log hom stacks $\HHom_S(X, Y)$ require $X \to S$ geometrically reduced. One could localize, apply \cite{loghomstackwise} to $m=0$ and check representability of the truncation maps separately. We instead follow the proof of \cite[Theorem 1.1.1]{evaluationspacefamiliesofstandardlogpoints}. We are not sure if $\ljsp{r}{X, m}$ is always representable by a log scheme.

\begin{proof}
%[Proof of Proposition \ref{prop:evstacksalllsch}]

We dealt with $X = \Aff^1, \Aff^n$ in the last section. If $X$ admits a strict map to $\Aff^n$, the pullback square
\[\begin{tikzcd}
\ljsp{r}{X, m} \ar[r] \ar[d] \lpbstrict         &\ljsp{r}{\Aff^n, m} \ar[d]      \\
\jsp{X, m} \ar[r]         &\jsp{\Aff^n, m}
\end{tikzcd}\]
shows representability. 

Strict \'etale maps $U \to X$ lead to pullback squares
\[\begin{tikzcd}
\ljsp{r}{U, m} \ar[r] \ar[d] \lpbstrict         &\ljsp{r}{X, m} \ar[d]        \\
U^\circ \ar[r]         &X^\circ.
\end{tikzcd}\]
To check the map $\ljsp{r}{X, m} \to X^\circ$ is representable by algebraic spaces, we can localize in $X$ to assume $X = \Spec A$ admits a global chart $Q \to A$. 

Write $Q$ as a coequalizer
\[\NN^a \rightrightarrows \NN^b \to Q\]
and write $X_a, X_b$ for the log schemes with log structure from the charts $\NN^a, \NN^b$. These admit strict maps to affine spaces, so we have representability of $\ljsp{r}{X_a, m}, \ljsp{r}{X_b, m}$. Obtain a representative of $\ljsp{r}{X, m}$ as the equalizer
\begin{equation}\label{eqn:eqrepability}
\ljsp{r}{X, m} \to \ljsp{r}{X_b, m} \rightrightarrows \ljsp{r}{X_a, m}
\end{equation}
among fs log schemes as in \cite[Proof of Theorem 1.1.1]{evaluationspacefamiliesofstandardlogpoints}.

If $X$ has charts Zariski locally, one can Zariski localize to apply the above argument. The functor $\ljsp{r}{X, m}$ then has a Zariski cover by log schemes and is itself a log scheme.

\end{proof}

\begin{remark}\label{rmk:nonfsrepability}

Define $\ljspnonfs{r}{X, m}$ to represent the functor
\[T \mapsto \Hom(T \times \ljet{r}{m}, X)\]
on \textit{all} log schemes $T$, not necessarily saturated, integral, or coherent. Remark \ref{rmk:nonfsevspA1} checks this coincides with the usual one $\ljspnonfs{r}{\Aff^n, m} = \ljsp{r}{\Aff^n, m}$ for $X = \Aff^n$. 

If $X$ is coherent, the proof of Theorem \ref{thm:repability} applies to show $\ljspnonfs{r}{X, m}$ is representable by a non fs log scheme. The equalizer \eqref{eqn:eqrepability} is computed in all log schemes and not fs ones, so it may differ. The map $\ljsp{r}{X, m} \to \ljspnonfs{r}{X, m}$ is fs-ification. 

\end{remark}

\subsection{More on the evaluation space $\ev X$}\label{ss:moreevsp}

All the components of $\ev X$ except the contact order zero one $X \subseteq \ev X$ map to the closed locus on $X^\circ$ where the log structure is supported:

\begin{remark}

Let $U \subseteq X$ be the possibly empty open locus where the log structure of $X$ is trivial. The preimage of $U$ under $\ev X \to X^\circ$ must be $\ev U = U$ because the log structure is trivial. All the components of $\ev X$ corresponding to nonzero contact orders lie over the closed complement $X \setminus U$. 

\end{remark}

\begin{example}\label{ex:evspA2}

See Figure \ref{ex:evspA2} for the components of $\ev \Aff^2 = \bigsqcup_{\NN^2} \Aff^2$. Write $\NN^+ = \NN \setminus 0$ for the monoid ideal of nonunits. We describe the map $\Aff^2 \to (\Aff^2)^\circ$ corresponding to each of the components indexed by $\NN \times \NN$:
\begin{itemize}
    \item[$(0, 0)$:] The identity $\Aff^2 \to (\Aff^2)^\circ$ on underlying schemes.
    \item[$\NN^+ \times 0$:] The projection $\Aff^2 \to \Aff^1$ composed with the embedding $\Aff^1 \to (\Aff^2)^\circ$ of the $y$-axis. 
    \item[$0 \times \NN^+$:] The projection $\Aff^2 \to \Aff^1 \to (\Aff^2)^\circ$ onto the $x$-axis.
    \item[$\NN^+ \times \NN^+$:] The projection to the origin $\Aff^2 \to 0 \in (\Aff^2)^\circ$.
\end{itemize}

To compute $P_{\NN^2}$, embed it strictly as the origin of $\Aff^2$. The evaluation space $\ev P_{\NN^2}$ is the fiber of $\ev \Aff^2 \to (\Aff^2)^\circ$ over the origin:
\begin{multicols}{2}
\begin{itemize}
    \item[$(0, 0)$:] $P_{\NN}$
    \item[$\NN^+ \times 0$:] $\Aff^1$
    \item[$0 \times \NN^+$:] $\Aff^1$
    \item[$\NN^+ \times \NN^+$:] $\Aff^2$. 
\end{itemize}
\end{multicols}

If $X \subseteq \Aff^2$ is the union of the two axes, compute the evaluation space $\ev X$ also as the pullback of $\ev \Aff^2$ along $X^\circ \subseteq (\Aff^2)^\circ$. 

\end{example}

\begin{figure}
    \centering
    \begin{tikzpicture}
    %The squares indexed by $\NN^2$
    \begin{scope}[shift={(-8, 0)}]
        \node[left] at (0, 4){$\ev \Aff^2$};
        \begin{scope}[shift={(0, 0)}]
        \filldraw[-, violet] (0, 0) -- (1, 0) -- (1, 1) -- (0, 1) -- (0, 0);
        \node[white] at (.5, .5){$(0, 0)$};
        \end{scope}
        \begin{scope}[shift={(2, 0)}]
        \filldraw[-, pattern color=blue, pattern=north east lines] (0, 0) -- (1, 0) -- (1, 1) -- (0, 1) -- (0, 0);
        \node at (.5, .5){$(1, 0)$};
        \end{scope}
        \begin{scope}[shift={(4, 0)}]
        \filldraw[-, pattern color=blue, pattern=north east lines]  (0, 0) -- (1, 0) -- (1, 1) -- (0, 1) -- (0, 0);
        \node at (.5, .5){$(2, 0)$};
        \end{scope}
        \begin{scope}[shift={(0, 2)}]
        \filldraw[-, pattern color=blue, pattern=north west lines]   (0, 0) -- (1, 0) -- (1, 1) -- (0, 1) -- (0, 0);
        \node at (.5, .5){$(0, 1)$};
        \end{scope}
        \begin{scope}[shift={(2, 2)}]
        \draw[-] (0, 0) -- (1, 0) -- (1, 1) -- (0, 1) -- (0, 0);
        \node at (.5, .5){$(1, 1)$};
        \end{scope}
        \begin{scope}[shift={(4, 2)}]
        \draw[-] (0, 0) -- (1, 0) -- (1, 1) -- (0, 1) -- (0, 0);
        \node at (.5, .5){$(2, 1)$};
        \end{scope}
        \begin{scope}[shift={(0, 4)}]
        \filldraw[-, pattern color=blue, pattern=north west lines]  (0, 0) -- (1, 0) -- (1, 1) -- (0, 1) -- (0, 0);
        \node at (.5, .5){$(0, 2)$};
        \end{scope}
        \begin{scope}[shift={(2, 4)}]
        \draw[-] (0, 0) -- (1, 0) -- (1, 1) -- (0, 1) -- (0, 0);
        \node at (.5, .5){$(1, 2)$};
        \end{scope}
        \begin{scope}[shift={(4, 4)}]
        \draw[-] (0, 0) -- (1, 0) -- (1, 1) -- (0, 1) -- (0, 0);
        \node at (.5, .5){$(2, 2)$};
        \end{scope}
    \end{scope}
    \begin{scope}[scale=.8]
        \begin{scope}[shift={(3, 3)}]
        \filldraw[-, violet] (-1, .5) -- (0, 0) -- (2, 0) -- (1, .5) -- (-1, .5);
        \draw[->] (.5, -.25) to (.5, -1.25);
        \draw[-] (-1, -1.5) -- (0, -2) -- (2, -2) -- (1, -1.5) -- (-1, -1.5);
        \end{scope}
        \begin{scope}[shift={(6, 3)}]
        \filldraw[-, pattern color=blue, pattern=north east lines] (-.5, 1) -- (-.5, 0) -- (1.5, 0) -- (1.5, 1) -- (-.5, 1);
        \draw[->] (.5, -.25) to (.5, -1.25);
        \draw[-] (-1, -1.5) -- (0, -2) -- (2, -2) -- (1, -1.5) -- (-1, -1.5);
        \draw[-] (-.5, -1.75) to (1.5, -1.75);
        \end{scope}
        \begin{scope}[shift={(0, 3)}]
        \filldraw[-, pattern color=blue, pattern=north east lines] (0, 1.5) -- (0, .5) -- (1, 0) -- (1, 1) -- (0, 1.5);
        \draw[->] (.5, -.25) to (.5, -1.25);
        \draw[-] (-1, -1.5) -- (0, -2) -- (2, -2) -- (1, -1.5) -- (-1, -1.5);
        \draw[-] (0, -1.5) to (1, -2);
        \end{scope}
    \end{scope}
    \end{tikzpicture}
    \caption{Left: The evaluation space $\ev \Aff^2$ consists of components $\Aff^2$ indexed by $\NN^2$. Right: The maps to the underlying scheme $(\Aff^2)^\circ$. 
    The purple component indexed by $c = 0$ maps identically to $(\Aff^2)^\circ$. The shaded blue components indexed by $\NN^+ \times 0$ and $0 \times \NN^+$ project onto the $y$- and $x$-axes, respectively. The blank white components indexed by $\NN^+ \times \NN^+$ map to the origin $0 \in (\Aff^2)^\circ$. }
    \label{fig:evspA2}
\end{figure}
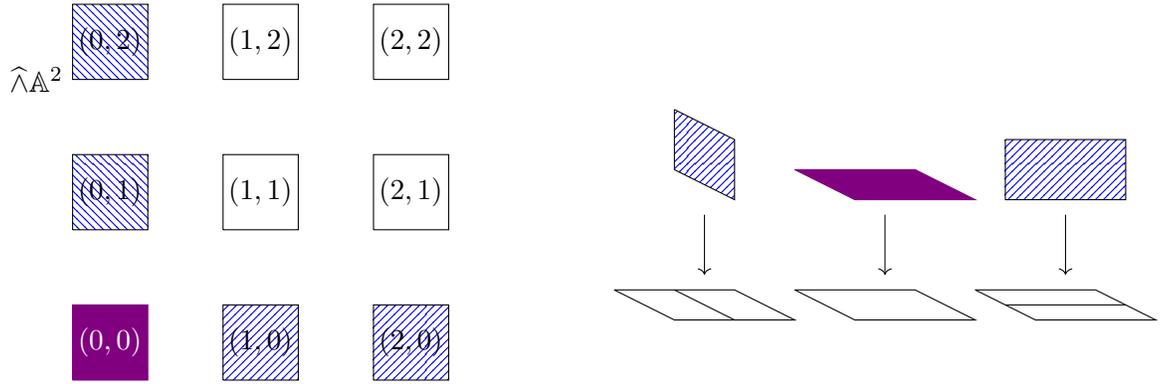

\begin{example}\label{ex:evspaffinetoricvar}

Let $R$ be another sharp fs monoid, yielding an affine toric variety
\[\Aff_R = \Spec k[R].\]
Our evaluation space $\ev \Aff_R$ has $T$-points given by maps of monoids
\[R \to \Gamma(M_T) \oplus \Gamma(\underline{\NN}).\]
This case runs parallel to $R = \NN$ and $\Aff_\NN = \Aff^1$. The map of monoids is a pair of 
\begin{itemize}
    \item a map $R \to \NN$ for each connected component of $T$, represented by $\bigsqcup_{\Hom(R, \NN)} \Spec k$, and
    \item a map $R \to \Gamma(M_T)$, represented by $\Aff_R$. 
\end{itemize}
Then $\ev \Aff_R = \bigsqcup_{\Hom(R, \NN)} \Aff_R$. This is why the components of $\ev X$ often look like the Artin fan of $X$, for example when $X$ is a toric variety. 

\end{example}

The above examples all had Zariski local charts, so they were represented by log schemes. Here's an example of what happens for log structures in the \'etale topology, where Artin fans can have monodromy.

\begin{example}

Let $D \subseteq \Aff^2$ be the nodal cubic curve
\[D : y^2 = x(x-1)^2.\]
Endow $X$ and $D$ with the divisorial log structure. Write $0 \in D \subseteq \Aff^2$ for the origin and $U = D \setminus 0$ for the complement, with underlying scheme isomorphic to $\GG_m$. The log structure $\bar M_D$ has stalks $\NN$ on $U = D \setminus 0$ and $\NN^2$ at the origin $0$, so $0 = P_{\NN^2}$. 

We compute $\ev D$. See Figure \ref{fig:evspnodalcubic}. Theorem \ref{thm:repability} only shows that $\ev D$ is representable by a log algebraic space because $D$ lacks Zariski local charts. In fact, $\ev D$ is a log scheme. 

Because $\ev D$ restricts to the evaluation spaces along the strict inclusions $U \subseteq D, 0 \in D$, the fiber over $0^\circ \in D^\circ$ is $\ev P_{\NN^2}$. Recall $\ev P_{\NN^2}$ has components $\Aff^2$ indexed by $\NN^+ \times \NN^+$, $\Aff^1$ for each of the lines $0 \times \NN^+$ and $\NN^+ \times 0$, and $P_{\NN^2}$ itself for $0 \times 0$. Likewise $\ev D \times_{D^\circ} U^\circ$ is $\ev U$. The strict map $U \to P_{\NN}$ exhibits $\ev U$ as the product of $\ev P_{\NN}$ with $U^\circ$, which has one component $U$ and the rest $U^\circ \times \Aff^1$. 

The space $\ev D$ is glued together from these spaces $\ev 0$ and $\ev U$. To check how they are glued together, one can work on the \'etale cover of $D$ by two $\Aff^1$'s joined at two nodes or simply a formal neighborhood of the origin $0$. In either case, the local picture coincides with the union of the axes in $\Aff^2$ mentioned in Example \ref{ex:evspA2}. 

The components are labeled by $0, 0 \times \NN^+$, and $\NN^+ \times \NN^+$: 
\begin{itemize}
    \item The component corresponding to contact order $0$ is $D$, as usual.
    \item Those corresponding to $\NN^+ \times \NN^+$ are all $\Aff^2$.
    \item The ones indexed by $0 \times \NN^+$ are $\Aff^2$. They factor through the normalization $\Aff^1 \to D$. 
\end{itemize}

\end{example}

\begin{figure}
    \centering
    \begin{tikzpicture}
    \begin{scope}[shift={(0, 3.5)}]
    \node[left] at (-1, 0){$E$};
	\draw[variable = \x, domain=-1:1.3] plot (\x*\x-1, .5*\x*\x*\x-.5*\x); 
	\draw[-] (-1, 0) to (-1, -.5);
	\draw[-] (.65, -.1) to (.65, .4);
    \end{scope}
    \begin{scope}[shift={(0, 3)}]
	\draw[variable = \x, domain=0:1.3] plot (\x*\x-1, .5*\x*\x*\x-.5*\x); 
	\draw[-] (.32, .18) to (.65, -.15) to (.65, -.65) to (0, 0);
    \end{scope}
    \begin{scope}[shift={(0, 1)}]
	\draw[variable = \x, domain=-1.3:1.3, blue] plot (\x*\x-1, .5*\x*\x*\x-.5*\x);
	\draw[->] (0, 1.5) to (0, .5);
	\node[left] at (-1, 0){$D$};
    \end{scope}
    \end{tikzpicture}
    \caption{The space $\ev D$ has a component $D$ for $c = 0$, components $\Aff^2$ mapping to the node $0 \in D$ for all $c \in \NN^+ \times \NN^+$, and then components like the above $E$ for $c \in \NN^+ \times 0$. The component $E = \Aff^2$ factors through the normalization $\Aff^1 \to D$ of the nodal curve. }
    \label{fig:evspnodalcubic}
\end{figure}
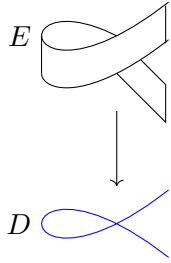

\begin{example}\label{ex:evsprootstack}

Let $f : \Aff^1 \to \Aff^1$ be the multiplication by $d$ map, for some integer $d > 0$. The induced map $\ev \Aff^1 \to \ev \Aff^1$ sends the a component corresponding to $q \in \NN$ to the component $dq \in \NN$. The map between the components $\Aff^1 \to \Aff^1$ is also multiplication by $d$. We see that $\ev -$ does not preserve surjectivity. 

\end{example}

\begin{example}\label{ex:evspblowup}

Let $B \to \Aff^2$ be the log blowup at the origin. Let $U, V \subseteq B$ be the charts containing $x' = x/y$ and $y' = y/x$. Both are isomorphic to $\Aff^2$, with evaluation spaces $\ev U, \ev V$ depicted in Figure \ref{fig:evspA2}. The maps $\ev U, \ev V \to \ev \Aff^2$ induced from the blowup are far from the identity. The component of $\ev U$ indexed by $(a, b)$ maps to the one indexed by $(a + b, b)$ in $\ev \Aff^2$. The map $\Aff^2 \to \Aff^2$ from the $(a, b)$-components with $a > b$ is the same as $U \to \Aff^2$, a chart for the blowup. The case of $V$ is the reverse. They glue together along the diagonal $(a, a)$ for $a \neq 0$. 

Write $\Delta \subseteq \NN^2$ for the positive diagonal $\Delta = \{(a, a) | a \in \NN^+\}$. Piece $\ev U, \ev V$ together along $\ev U \cap V$ to see that $\ev B$ with its map to $\ev \Aff^2$ consists of:
\begin{itemize}
    \item a contact order-zero component $B$,
    \item components $\Aff^2$ projecting onto the strict transforms of the $x$- and $y$-axes in $B$ indexed by $\NN^+ \times 0$ and $0 \times \NN^+$, 
    \item the majority of the components $\NN^+ \times \NN^+ \setminus \Delta$ excepting the diagonal of $\Aff^2$'s mapping to the corresponding $\Aff^2 \subseteq \ev \Aff^2$ as either the $U$ or $V$ chart for the blowup, 
    \item the diagonal $\Delta \subseteq \NN^2$ consisting of $B$ mapping to $\Aff^2$ as the blowup of the origin. 
\end{itemize}

Write $P_{\NN^2} \in \Aff^2$ for the origin and $E \subseteq B$ for the exceptional divisor. A similar picture arises for $\ev E \to \ev P_{\NN^2}$. There are components indexed by $\NN^2$, with the axes $\NN^+ \times 0$ and $0 \times \NN^+$ given by $\Aff^1$'s. The origin $(0, 0)$ is $E$ as usual. The components along the diagonal $\Delta \subseteq \NN^2$ are $B$, mapping to the corresponding $\Aff^2$'s of $\ev P_{\NN^2}$ as the log blowup of the origin. The rest of the components $\NN^+ \times \NN^+ \setminus \Delta$ are $\Aff^2$'s, mapping to the $\Aff^2$'s of $\ev P_{\NN^2}$ as charts of blowups.

\end{example}

\begin{remark}\label{rmk:evspnotsurj}

Even if $X \to Y$ is proper or surjective, the map on evaluation spaces $\ev X \to \ev Y$ need not be. Examples \ref{ex:evsprootstack} and \ref{ex:evspblowup} both depict surjective maps that lead to non-surjective maps after taking $\ev -$. The multiplication by $d$ map even misses most of the components of the target. For the blowup $B \to \Aff^2$, restrict to the exceptional divisor $E \to P_{\NN^2}$. Take the identity $P_{\NN^2} \longequals P_{\NN^2}$, thought of as a point $P_{\NN} \to \ev P_{\NN^2}$. This does not factor through the blowup.

The latter Example \ref{ex:evspblowup} takes the proper log blowup of the origin $B \to \Aff^2$ and produces components of $\ev B$ that map to $\ev \Aff^2$ as charts of the blowup $B \to \Aff^2$, not proper.

\end{remark}

Log modifications $X \to Y$ do not lead to log modifications $\ev X \to \ev Y$ by Remark \ref{rmk:evspnotsurj}. Less is true. 

\begin{lemma}\label{lem:evspmaplogetmonom}

If $X \to Y$ is log \'etale or a monomorphism, so are $\ev X \to \ev Y$ and $\ljsp{r}{X, m} \to \ljsp{r}{Y, m}$. 

\end{lemma}

\begin{proof}

Omitted.

\end{proof}

\begin{example}[Toric singularity]

Consider the cone $Q$ generated by $(1, 1), (1, 0), (1, -1)$ in $\ZZ^2$ and associated affine toric variety
\[\Aff_Q \coloneqq \Spec k[x, y, z]/(xy=z^2)\]
with natural log structure from the toric divisors. This toric variety is log smooth, so its log jet space $\ljsp{r}{\Aff_Q, m}$ is a log smooth log scheme. 

A morphism $f : \Aff_Q \to \Aff_P$ of affine toric varieties associated to a morphism $P \to Q$ of fs sharp monoids is log smooth if and only if the kernel and the torsion subgroup of the cokernel of $\gp P \to \gp Q$ are finite groups of orders invertible in $k$ \cite[Theorem 3.5]{katooriginal}. The map $f$ is log \'etale if the kernel and full cokernel of $\gp P \to \gp Q$ satisfy that condition. 

The map $\NN^2 \to Q$ sending the generators to $(1, 1)$ and $(1, -1)$ thus induces a log \'etale morphism of toric varieties $\Aff_Q \to \Aff^2$. The log jet spaces are thus pulled back
\[\begin{tikzcd}
\ljsp{r}{\Aff_Q, m} \ar[r] \ar[d] \lpbstrict         &\ljsp{r}{\Aff^2, m} \ar[d]        \\
\ev \Aff_Q \ar[r]          &\ev \Aff^2.
\end{tikzcd}\]

Remark that $\Aff_Q$ is the pullback
\[\begin{tikzcd}
\Aff_Q \ar[r] \ar[d] \lpbstrict      &\Aff^1 \ar[d, "\stquot{2}"]         \\
\Aff^2 \ar[r, "m"]      &\Aff^1
\end{tikzcd}\]
of the multiplication map $m(x, y) = xy$ along the squaring map $\stquot{2}$. The map $\ev \stquot{2} : \ev \Aff^1 \to \ev \Aff^1$ hits only even components and is the squaring map there. The multiplication map $m$ sends the $(a, b)$ components of $\ev \Aff^2$ to the $a + b$ component of $\ev \Aff^1$, each map induced by $m$ itself on the components. Pulling back, we describe $\ev \Aff_Q$ with components indexed by pairs $(a, b) \in \NN^2$ that add to an even number $2 | a + b$, each of which is $\Aff_Q$. The space $\ev \Aff_Q$ is log smooth, and the log jet spaces $\ljsp{r}{\Aff_Q, m}$ are strict affine bundles over it. 

This coincides with the description $\ev \Aff_Q = \bigsqcup_{\Hom(Q, \NN)} \Aff_Q$ of Example \ref{ex:evspaffinetoricvar}.

\end{example}

\begin{remark}\label{rmk:Qevsp}

One can define evaluation spaces $\ev_Q X$ for any sharp fs monoid $Q$. The standard log point $P_{\NN}$ is replaced by the point $P_Q$, or $\Spec k$ with log structure
\[M_{P_Q} \coloneqq k^* \oplus Q.\]
The evaluation spaces of affine space are then
\[\ev_Q \Aff^1 = \bigsqcup_Q \Aff^1, \qquad \ev_Q \Aff^n = \bigsqcup_{Q^n} \Aff^n.\]
The examples in this section can be rewritten in this level of generality. 

\end{remark}

\subsection{Log jet spaces in log Gromov--Witten theory}\label{ss:logjetloggw}

The evaluation space $\ev X$ is almost the same as the evaluation stack $\wedge X$. Most of this section requires familiarity with log Gromov--Witten theory at the level of \cite{loggw}.

Recall the construction of the evaluation stack $\wedge X$ of \cite{evaluationspacefamiliesofstandardlogpoints}. Let $\scr P$ be the origin of $\stquot{\Aff^1/\GG_m}$, with underlying stack $B\GG_m$ and rank-one log structure. The map $\scr P \to \scr P^\circ = B\GG_m$ to the underlying stack is the universal ``family of standard log points'' by loc. cit \S 2.3. If $T \to B \GG_m$ classifies a line bundle $L \to T$, give $L$ the divisorial log structure at the zero section $0 : T \to L$. The family of standard log points is the log structure pulled back from $0 : T \to L$.

The evaluation stack $\wedge X$ has functor of points 
\[\wedge X(T) \coloneqq \Hom(T \times_{B\GG_m} \scr P, X)\]
on fs log schemes $T$. The strict pullback square
\[\begin{tikzcd}
P_{\NN} \ar[r] \ar[d]\lpbstrict       &\scr P \ar[d]         \\
\Spec k \ar[r]         &B\GG_m
\end{tikzcd}\]
yields equality
\[\wedge X \times_{B\GG_m} \Spec k = \ev X.\]
Because $\Spec k \to B\GG_m$ is a strict $\GG_m$-torsor, the same holds for $\ev X \to \wedge X$.

\begin{proposition}\label{rmk:compareevalspabram}
The map $\ev X \to \wedge X$ is a strict torsor for $\GG_m$. 
\end{proposition}

\begin{remark}

One can likewise define a non-fs version of $\wedge X$ or allow general fs sharp monoids $Q$, replacing $\scr P$ with the origin of $\af[Q] = \stquot{\Aff_Q/\Aff_{\gp Q}}$ over $B \Aff_{\gp Q}$. One reason the original evaluation space $\wedge X$ requires caution with all log schemes is that the quotient map $M_T \to \bar M_T$ is likely not a torsor for $\OO^*_T$. One can replace $\bar M_T$ by the stack quotient to retain some features similar to the case of integral log schemes. Avoid this discussion by defining ``families of log points'' as pullbacks of the universal object
\[\scr P_Q \to B \Aff_{\gp Q}\]
instead. 

\end{remark}

The evaluation stack $\wedge X$ was invented as a target for the evaluation maps in log Gromov--Witten theory. Let $\Ml(X)$ be the moduli space of log stable maps \cite{loggw}. Over an fs log scheme $T$, it parameterizes families $C \to T$ of log smooth curves of genus $g$ (i.e., nodal curves) with $n$ marked points together with a map $C \to X$ satisfying a stability condition. 

The map $\Ml(X) \to (\wedge X)^n$ restricts a map $C \to X$ to its $n$ marked points. Since these have log structure, this gives a family of log points mapping to $X$. 

Étale-locally in $C \to T$ near a marked point, the curve is isomorphic to $\Aff^1_T \to T$. Restricting the stable map $C \simeq \Aff^1_T \to X$ to the neighborhood of the origin $\ljet{1}{m}\times T \subseteq \Aff^1_T$ gives a $T$-valued log jet of $X$ for any $m$, including $m = \infty$. This almost gives a map
\begin{equation}\label{eqn:evmaptoarcsp}
\Ml(X) \dashrightarrow (\ljsp{r}{X, \infty})^n   
\end{equation}
factoring the evaluation map. The subtlety is the same $\GG_m$ quotient as $\ev X \to \wedge X$.

Write $\tljet{1}{m} \subseteq \stquot{\Aff^1/\GG_m}$ for the strict $m$th infinitesimal neighborhood of the origin, isomorphic to the quotient of $\ljet{1}{m}$ by the action of $\GG_m$. Write $\tljsp{1}{X, m}$ for the stack of maps
\[\tljsp{1}{X, m}(T) \coloneqq \Hom(T \times_{B\GG_m} \tljet{1}{m}, X)\]
from these stacky log jets to $X$. The strict pullback square 
\[\begin{tikzcd}
\ljsp{1}{X, m} \ar[r] \ar[d] \lpbstrict    &\tljsp{1}{X, m} \ar[d]    \\
\Spec k \ar[r]         &B\GG_m
\end{tikzcd}\]
again shows representability of $\tljsp{1}{X, m}$ by a log algebraic stack $\stquot{\ljsp{1}{X, m}/\GG_m}$. The reader can generalize the construction to create $\tljsp{r}{X, m}$ for $r > 1$. 

The rest of this section proves the evaluation maps factor through this stacky log jet space. 

\begin{theorem}\label{thm:jetspevmap}
The map from log stable maps to the evaluation stack factors through the stacky log jet space
\[\Ml(X) \to (\tljsp{1}{X, m})^n \to (\wedge X)^n\]
for any $m$, in particular $m = \infty$. 
\end{theorem}

The reason to take the $\GG_m$ quotient is that nontrivial families of curves $C \to T$ could restrict to nontrivial families of log points. These families arise from automorphisms of the log point $P_\NN = \ljet{1}{0}$. Another way to say this is that a curve $C$ with marked point $p$ is isomorphic to $\Aff^1$ with the origin, but there is a noncanonical choice of isomorphism. Different choices of isomorphism with $0 \in \Aff^1$ differ by scaling by $\GG_m$.

The automorphism group scheme of $\ljet{1}{0}$ is $\GG_m$. The same is not true of $\ljet{1}{m}$. 

\begin{lemma}
An automorphism of $T \times \ljet{1}{m}$ over $T$ consists of a unit\footnote{
This is saying that the log automorphism group scheme $\Aut(\ljet{1}{m})$ is the \emph{Weil restriction} $\WR_{\jet{m}/\Spec k} \GG_m$ (Section \S \ref{s:logtansplogjetsp}). 
}
\[u \in \GG_m(T \times \ljet{1}{m})\]
acting via multiplication. 
\end{lemma}

\begin{proof}
Omitted. 
\end{proof}

If $T = \Spec A$ with trivial log structure for example, the automorphism groups 
\[\Aut(A \adj{t}_m) = A^* + t A \adj{t}_m, \qquad \Aut(A) = A^*\]
differ. These could lead to more complicated families of log jets, necessitating the stack quotient of $\ljsp{1}{X, m}$ by the larger group $\GG_m(T \times \ljet{1}{m})$ for each $T$. Here is an example. 

\begin{example}
Let $k$ be a field of characteristic two. The element
\[1 + t^m \in \GG_m(k \adj{t}_m)\]
squares to the identity, so it lies in $\mu_2(k \adj{t}_m)$. The subgroup scheme generated by $1 + t^m$ is isomorphic to $\mu_2$. 

The Kummer sequence
\[1 \to \mu_2 \to \GG_m \overset{\stquot{2}}{\longrightarrow} \GG_m \to 1\]
is exact as fppf sheaves but not étale sheaves because the characteristic is two. The product
\[X = \GG_m \times \ljet{1}{m}\]
has two actions by $\mu_2$, one on either factor. Write $Z$ for the quotient of $X$ by the diagonal action
\[\mu \times X \to X; \qquad \zeta.(u, f(t)) = (\zeta \cdot u, (1 + t^m)\cdot f(t)).\]
The composite $X \to \GG_m \overset{\stquot{2}}{\longrightarrow} \GG_m$ descends to $Z$, giving a family $Z \to \GG_m$ which is isomorphic to $\GG_m \times \ljet{1}{m} \to \GG_m$ locally in the fppf topology. 

\end{example}

Fortunately, these families do not show up as neighborhoods of marked points of families of curves. Let $C \to T$ be a family of curves with a marked point $p : T \dashrightarrow C$. If $I$ is the ideal of the marked point, write 
\[\begin{split}
D_m &\coloneqq \SSpec C {\OO_C/I^{m+1}}   \\
D &\coloneqq \SSpec C {\left(\OO_C\right)_I^\wedge}
\end{split}\]
for the infinitesimal neighborhoods of the marked point $p$ in $C$. The symbol ${-}^\wedge_I$ refers to completion at $I$. 

Our main theorem \ref{thm:jetspevmap} will result from this proposition.

\begin{proposition}\label{prop:linebundleinfnbd}
Use notation $D_m, C, p, T$ as above. The map $D_m \to T$ is isomorphic to the infinitesimal neighborhood of the zero section of a line bundle $L$ over $T$. 
\end{proposition}

The curve $C$ is isomorphic to $\Aff^1_T$ in a neighborhood of $p$. This means there is a diagram 
\[
\begin{tikzcd}
C' \ar[r] \ar[d]      &C \ar[d]      \\
T' \ar[r]      &T
\end{tikzcd}
\]
with $T'\to T$ and $C'\to C \times_T T'$ étale maps and an isomorphism $C'\simeq \Aff^1_{T'}$ such that the zero section agrees with $p$ on $T'$ \cite[pg. 222]{fkatologdefthy}, \cite[Theorem 1.1]{loggw}. The $D_m$ are then isomorphic to $\ljet{1}{m} \times T'$ over $T'$. 

We first show that this can be done for $D_m$ by localizing only in $T$.

\begin{lemma}\label{lem:finiteetaleisom}
Let $E \to \Spec A$ be a proper map admitting a section $\sigma : \Spec A \dashrightarrow E$ that is a universal homeomorphism. Suppose $f : \ljet{1}{m} \times \Spec A \to E$ is étale and surjective and that it sends the zero section to $\sigma$. 
%it restricts to an isomorphism over the section $\sigma$. 
Then $f$ is an isomorphism. 

\end{lemma}

\begin{proof}

We first argue that $f$ induces an isomorphism of the zero section with $\sigma$. Write $U \subseteq \ljet{1}{m} \times \Spec A$ for the closed pullback of the section $\sigma$ along $f$
\[
\begin{tikzcd}
U \ar[r] \ar[d] \pb       &\Spec A \ar[d]        \\
\ljet{1}{m} \times \Spec A \ar[r]      &E.
\end{tikzcd}
\]
The map $U \to \Spec A$ is an étale cover with a necessarily open section $\Spec A \dashrightarrow U$ inducing the zero section of $\ljet{1}{m} \times \Spec A$. The closed complement $U \setminus \Spec A$ is a closed subset of $\ljet{1}{m} \times \Spec A$ that does not meet the zero section, hence the empty set. Thus $U \longequals \Spec A$. 

The map $f$ is finite étale because $\ljet{1}{m} \times \Spec A$ is proper over $\Spec A$. It factors through the trivial finite étale cover $E \longequals E$. These two finite étale covers coincide over the closed section $\sigma$. By topological invariance of the étale site \cite[04DZ]{sta}, they are isomorphic.

\end{proof}

\begin{corollary}\label{cor:finiteetaleisom}
Use notation as above. If $\ljet{1}{\infty} \times T \to D$ is an étale surjection sending the zero section to the marked point $p : T \to D$, it is an isomorphism. 
\end{corollary}

\begin{proof}

The $m$th infinitesimal neighborhoods are all isomorphic by Lemma \ref{lem:finiteetaleisom}. The isomorphism over each finite $m$ results in the infinite case $\ljet{1}{\infty} \times \Spec A \simeq D$ by applying Bhatt's theorem \cite[Theorem 4.1, Remark 4.6]{bhatttannakaalgn}. 

\end{proof}

The automorphism group scheme of the ordinary affine line $(\Aff^1)^\circ$ consists of affine transformations $\GG_m \ltimes \Aff^1$. The automorphisms compatible with the log structure of $\Aff^1$ are $\GG_m$. We use this to constrain the automorphisms of infinitesimal neighborhoods of marked points arising in families.

There exists an étale surjection $T' \to T$ and an isomorphism $w : D \times_T T' \simeq \ljet{1}{\infty} \times T'$ over $T'$. Pull this isomorphism back both possible ways to $T'' \coloneqq T'\times_T T'$ to get an automorphism
\[\tilde w : \ljet{1}{\infty} \times T'' \simeq D \times_T T'' \simeq \ljet{1}{\infty} \times T''.\]

\begin{lemma}\label{lem:Gmtwists}
Suppose $D$ is the infinitesimal neighborhood of a marked point $p : T \dashrightarrow C$ of a curve as above. There is a choice of $w : D \times_T T' \simeq \ljet{1}{\infty} \times T'$ such that the automorphism $\tilde w$ lies in $\OO_{T''}^*$. That is, it lies in the subgroup
\[\GG_m(T'') \subseteq \GG_m(T'' \times \ljet{1}{m})\]
of the group of all automorphisms of $T'' \times \ljet{1}{m}$. 
\end{lemma}

\begin{proof}

Let $\hat C \to C$ be an étale cover with an isomorphism $\hat C \simeq \Aff^1_{T}$. Write $\hat D \subseteq \hat C$ for the infinitesimal neighborhood of the zero section. The induced map $\hat D \to D$ is an isomorphism by Corollary \ref{cor:finiteetaleisom}. Choose $w$ to factor the composite
\[D \simeq \hat D \subseteq \hat C \simeq \Aff^1_{T}.\]

Then $\tilde w$ fits into a diagram
\[\begin{tikzcd}
\ljet{1}{m} \times T'' \ar[r] \ar[d]          &\hat D \times_T T'' \ar[d] \ar[r]        &\ljet{1}{m} \times T'' \ar[d]      \\
\Aff^1_{T''} \ar[r]    &\hat C \times_T T'' \ar[r]        &\Aff^1_{T''}.
\end{tikzcd}\]
But automorphisms of $\Aff^1$ with log structure are $\GG_m$. 

\end{proof}

\begin{proof}[Proof of Proposition \ref{prop:linebundleinfnbd}]

The line bundle $L$ is obtained by gluing $\Aff^1_{T'}$ along the isomorphism extending $\tilde w$ obtained in the proof of Lemma \ref{lem:Gmtwists}. Étale-local line bundles are the same as Zariski-local line bundles. 
    
\end{proof}

\begin{proof}[Proof of Theorem \ref{thm:jetspevmap}]

The universal infinitesimal neighborhood of the zero section in a line bundle is $\tljet{1}{m} \to B\GG_m$. Proposition \ref{prop:linebundleinfnbd} shows any $D \to T$ induced from the infinitesimal neighborhood of a marked point $p : T \dashrightarrow C$ of a curve $C$ is isomorphic in $T$ to the infinitesimal neighborhood of a line bundle $L$ over $T$. The restriction of a map $C \to X$ to $D \to X$ is therefore an element of the Weil restriction along $\tljet{1}{m} \to B\GG_m$ of the target $X$, i.e. the stacky log jet space
\[\tljsp{1}{X, m} = \HHom_{B\GG_m}(\tljet{1}{m}, X).\]
    
\end{proof}

\begin{remark}[Thanks to Sam Molcho]
    
Consider the analogous map $\Psi : \Ms(X) \to (\tjsp{X, \infty})^n$ without log structure. Its image is a closed substack of ``integrable'' jets, in the sense that they extend to curves. 

The map $\Psi$ is not injective, for example if $n = 0$. It should however restrict to an injection on the locally closed strata of $\Ms(X)$ where each component of the source curve $C$ has a marked point. 

To see this, replace $X$ by an ambient projective space $X \subseteq \PP^N$. Assume no component or marked point of $C$ maps into a coordinate hyperplane of $\PP^N$ by change of coordinates. Pull back $\Aff^N \subseteq \PP^N$ to get a holomorphic map $U \to \Aff^N$ from a dense open $U \subseteq C$ containing all the marked points. By elementary complex analysis, a holomorphic function is determined by its analytic expression. The map $U \to \PP^N$ is then uniquely determined by its arc at each marked point. Extending to the entire curve $U \subseteq C$ is immediate. 

%The closure in $\tjsp{X, m}$ of these locally closed strata gives another compactification of stable maps. One must use caution because $\tjsp{X, m}$ is not proper, but the open representable substack $D \subseteq \tjsp{X, m}$ of stacky nonzero jets is proper over $X$. Maps $C \to X$ that do not contract any components factor through $D$. It is unlikely that the enumerative geometry of $\Ms(X)$ coincides with this closure in $\tjsp{X, m}$ for dimension reasons, however. 

\end{remark}

\subsection{Log jets of exotic objects}\label{ss:logjetsexotic}

Functors $X$ on fs log schemes which admit a log \'etale cover by a log scheme $Y \to X$ are called ``logarithmic spaces.'' The main example is the log ``multiplicative group''
\[\begin{split}
\Glog(S) &\coloneqq \Gamma(\gp{M}_S)         
\end{split}.\]
The map $\PP^1 \to \Glog$ of the element $x-y \in \Gamma(\gp M_{\PP^1})$ is a log \'etale cover \cite[Lemma 2.2.7.3]{logpic}, but $\Glog$ is not itself representable by a log scheme \cite[Lemma 2.2.7.2]{logpic}. The fan of the toric variety $\PP^1$ is a line subdivided at the origin, whereas the ``fan'' of $\Glog$ is the line without subdivision at the origin. 

The related tropical multiplicative group 
\[\Gtrop (S) \coloneqq \Gamma(\gp{\bar M}_S)\]
is the stack quotient of $\Glog$ by the action of $\GG_m$. It has a log étale cover by $\stquot{\PP^1/\GG_m}$.

\begin{remark}

We cannot use Lemma \ref{lem:evspmaplogetmonom} to deduce the log jet space $\ljsp{r}{\Glog, m}$ of $\Glog$ is representable by a logarithmic space. The log blowup $\PP^1 \to \Glog$ need not lead to a cover $\ljsp{r}{\PP^1, m} \to \ljsp{r}{\Glog, m}$, as in the non-surjective Examples \ref{ex:evsprootstack} and \ref{ex:evspblowup}. We instead show representability directly. 

\end{remark}

\begin{lemma}\label{lem:logjetstorsors}

Let $G$ be a group scheme with trivial log structure and $P \to X$ a strict map which is a strict \'etale torsor for $G$. The log jet space $\ljsp{r}{G, m}$ has a group structure. Furthermore, the map on log jet spaces $\ljsp{r}{P, m} \to \ljsp{r}{X, m}$ is a torsor for the sheaf of groups $\ljsp{r}{G, m}$. 

\end{lemma}

We impose strictness just to avoid comparing torsors for ordinary and fs pullbacks. 

\begin{proof}

This results formally from the compatibility of log jet spaces $\ljsp{r}{-, m}$ with limits in Remark \ref{rmk:logjetspacepb}.

\end{proof}

\begin{proposition}

The log jet spaces $\ljsp{r}{\Glog, m}$ and $\ljsp{r}{\Gtrop, m}$ of the log and tropical ``multiplicative groups'' $\Glog, \Gtrop$ are representable by logarithmic spaces. In fact,
\begin{equation}\label{eqn:logjetspGtrop}
\ljsp{r}{\Gtrop, m} = \bigsqcup_\ZZ \Gtrop,    
\end{equation}
and the log jet space of $\Glog$ forms a $\ljsp{r}{\GG_m, m}$-torsor over this space. 

\end{proposition}

\begin{proof}

Lemma \ref{lem:logjetstorsors} addresses the claim for $\Glog$, and it remains to show \eqref{eqn:logjetspGtrop} holds. 

A point $T \to \ljsp{r}{\Gtrop, m}$ is a section $\Gtrop(T \times \ljet{r}{m})$. But the characteristic monoid of $T \times \ljet{r}{m}$ has sections
\[\Gamma(\bar M_{T \times \ljet{r}{m}}) = \Gamma(\bar M_T) \oplus \Gamma(\underline{\ZZ}).\]
Compare with Example \ref{ex:evspA1} of $\Aff^1$. This functor is representable by
\[\Gtrop \times \bigsqcup_\ZZ \Spec k = \bigsqcup_\ZZ \Gtrop.\]

\end{proof}

We do not know if the log jet space of an arbitrary logarithmic space will be representable by a logarithmic space.

An \emph{Artin cone} is a log algebraic stack $\af[P]$ of the form
\[\af[P] \coloneqq \stquot{\Aff_P / \Aff_{\gp P}},\]
with natural log structure making the quotient map $\Aff_P \to \af[P]$ strict. Artin stacks $\scr B$ that admit \'etale covers by Artin cones are called \emph{Artin fans}. We recall them here but direct the reader to \cite{wisebounded}, \cite{rendimentodeicontiwisepandharipandeherrmymolcho} for a proper introduction. These are alternatives to monoschemes \cite{ogusloggeom}, Kato fans \cite{kato94}, or the original tropicalizations of toroidal spaces \cite{kempfknudsenmumfordsaintdonatkkmstoroidalembeddings} reviewed in \cite[Appendix B]{loggw}. 

The point of Artin fans is that every fs log scheme $X$ admits a canonical strict map to an Artin fan \cite{wisebounded}, called the \emph{tropicalization} of $X$
\[X \to \af[X].\]

\begin{example}

If $X$ is a toric variety with dense torus $T$, the Artin fan is the quotient $X \to \stquot{X/T}$. The points of $\stquot{X/T}$ are in bijection with the cones of the fan of $X$. If the dual monoid $\Hom(P, \NN)$ is a cone of the fan of $X$, the corresponding point is the stack $B\gp P$. 

\end{example}

\begin{example}

If $(X, D)$ is a pair of a smooth scheme $X$ and s.n.c. divisor $D$, the tropicalization looks like the ``cone over the intersection complex of $D$.'' At each stratum of $X$ given by an intersection $\bigcap_I D_i$ of $n$ components $D_i \subseteq D$ of $D$, one takes a copy of $\af^n$. These are glued together along specialization maps -- an inclusion $J \subseteq I$ of components leads to a map $\af^{\#I} \to \af^{\# J}$.  

\end{example}

Recall M. Olsson's stack of fs log structures $\Log$
A map $T \to \Log$ of schemes parameterizes an fs log structure $M_T$ on the scheme $T$. The stack $\Log$ carries an \'etale cover by all Artin cones $\af[P]$ ranging over all fs monoids $P$.

\begin{example}
The map $\af[P] \to \Spec k$ is log \'etale. Lemma \ref{lem:letpb} ensures a pullback square
\[\begin{tikzcd}
\ljsp{r}{\af[P], m} \ar[r] \ar[d] \lpbstrict        &\ljsp{r}{\Spec k, m}  \ar[d, equals]      \\
\ev \af[P] \ar[r]      &\ev \Spec k
\end{tikzcd}\]
that identifies $\ljsp{r}{\af[P], m} \simeq \ev \af[P]$. The same argument with $\Log$ in place of $\af[P]$ shows $\ljsp{r}{\Log, m} \simeq \ev \Log$. These evaluation spaces have functor of points on fs log schemes $T$
\[
\begin{split}
\ev \af[P](T) &\coloneqq \Hom(P, \Gamma(\overline{M}_T) \oplus \NN) = \bigsqcup_{\Hom(P, \NN)} \af[P] (T)  \\
\ev \Log (T) &\coloneqq \{\text{another log structure }M'_T \text{ on }T, \text{ with a map }M'_T \to M_T \oplus \NN\}.
\end{split}
\]

\end{example}

\begin{remark}\label{rmk:logstackpbsquare}
Since every log stack admits a strict map $X \to \Log$, we have pullback squares
\[\begin{tikzcd}
\ljsp{r}{X, m} \ar[r] \ar[d] \lpbstrict      &\ljsp{r}{\Log, m} \simeq \ev \Log  \ar[d]     \\
\jsp{X^\circ, m} \ar[r]        &\jsp{\Log^\circ, m}
\end{tikzcd}\]
We prove the ordinary jet spaces of algebraic stacks are themselves algebraic stacks in the appendix. From this square, Corollary \ref{cor:ljetrepablestacks} concludes that the log jet spaces of log algebraic stacks are representable by log algebraic stacks.

\end{remark}

\subsection{Bhatt's theorem}\label{ss:bhattthm}

We now adapt B. Bhatt's theorem describing the arc space as a limit of the jet spaces. In \cite[Theorem 4.1, Remark 4.6]{bhatttannakaalgn}, Bhatt proved the arc space is the limit of the truncation maps
\begin{equation}\label{eqn:bhattthm}
\jsp{X, \infty} = \lim_m \jsp{X, m}    
\end{equation}
for all schemes $X$. Corollary 1.5 of \cite{bhatttanaka2} strengthened this result to hold in particular for $X$ a noetherian Artin stack with affine diagonal. Artin fans are locally noetherian Artin stacks with affine diagonal, so we obtain Bhatt's theorem \eqref{eqn:bhattthm} for log jets by working with Artin fans.

\begin{proposition}\label{prop:bhattthm}

Suppose $X$ is an fs log scheme. Then
\[\ljsp{r}{X, \infty} = \lim_m \ljsp{r}{X, m}.\]
This identification is functorial in $X$. 

\end{proposition}

\begin{proof}

\begin{comment}
As in the classical case, the map $\ljsp{r}{X, \infty} \to \lim_m \ljsp{r}{X, m}$ is a proper (strict) monomorphism and it remains to show it is surjective -- that any compatible system of jets $\{\ljet{r}{m} \to X\}_{m \in \NN}$ comes from an arc $\ljet{r}{\infty} \to X$. Consider such a compatible system. 
\end{comment}

Write $\af[X]$ for the Artin fan of $X$. For each $m \in \NN \cup \{\infty\}$, the strict map $X \to \af[X]$ ensures a pullback square
\[\begin{tikzcd}
\ljsp{r}{X, m} \ar[r] \ar[d] \lpbstrict     &\ev \af[X] \ar[d]     \\
\jsp{X^\circ, m} \ar[r]      &\jsp{\af[X]^\circ, m}
\end{tikzcd}.\]
We identified $\ev \af[X] \simeq \ljsp{r}{\af[X], m}$. These squares are compatible with the truncation maps. 

Take the limit of these squares over $m$ and apply Bhatt's theorem \cite[Corollary 1.5]{bhatttanaka2} to identify the bottom row with 
\[\jsp{X^\circ, \infty} \to \jsp{\af[X]^\circ, \infty}.\]
Then the fiber products of the squares are the same
\[\ljsp{r}{X, \infty} = \lim_m \ljsp{r}{X, m}.\]

\end{proof}

\subsection{Relative log jets}

Let $X \to S$ be a map of log schemes. Define a functor $\ljsp{r}{X/S, m}$ on affine fs log schemes $\Spec A$ by
\[
\ljsp{r}{X/S, m}(\Spec A) \coloneqq 
\left\{\begin{tikzcd}
\Spec A \times_k \ljet{r}{m} \ar[r] \ar[d]        &X \ar[d]      \\
\Spec A \ar[r]        &S.
\end{tikzcd}\right\}\]
If $m = \infty$, one defines relative log arcs $\ljsp{r}{X/S, \infty}$ by completing 
\[\Spec A \hat{\times}_k \ljet{r}{\infty} = \Spec A \adj{t}.\]

Working relative to a base is important in log geometry. Log smooth curves can attain singularities only over a base with nontrivial log structure for example. We prove analogues of our results in the relative setting and conclude with examples.

Any point $\Spec A \to S$ can be thought of as a constant jet 
\[\Spec A[t]/t^m \to \Spec A \to S,\]
giving a map $S \to \ljsp{r}{S, m}$. Relative log jets are the fs fiber product
\[\ljsp{r}{X/S, m} = \ljsp{r}{X, m} \times^\ell_{\ljsp{r}{S, m}} S.\]
In particular, relative log jets are representable by a log algebraic space. If $S = \Spec k$ is a point with trivial log structure, $\ljsp{r}{S, m} = S$ and $\ljsp{r}{X/S, m} = \ljsp{r}{X, m}$. 

Given a pair of maps $X \to Y \to Z$, there is an fs pullback diagram
\[\begin{tikzcd}
\ljsp{r}{X/Y, m} \ar[r] \ar[d] \lpb        &\ljsp{r}{X/Z, m} \ar[d]       \\
Y \ar[r]       &\ljsp{r}{Y/Z, m}.
\end{tikzcd}\]    

For $m=1$, the relative jet spaces are relative tangent bundles and this pullback is dual to the exact sequence of K\"ahler differentials
\[\lkah{Y/Z}|_X \to \lkah{X/Z} \to \lkah{X/Y} \to 0.\]
Exactness on the left results from smoothness of $X \to Y$, and this guarantees surjectivity of $\ljsp{r}{X/Z, m} \to \ljsp{r}{Y/Z, m}|_X$ more generally. 

For $m=0$, this gives a ``relative evaluation space'' $\ev_S X$ obtained by pulling back $\ev X \to \ev S$ along the inclusion $S \subseteq \ev S$ of the contact order-zero component
\[\begin{tikzcd}
\ev_S X \ar[r] \ar[d] \lpbstrict         &\ev X \ar[d]      \\
S \ar[r]       &\ev S.
\end{tikzcd}\]

\begin{proposition}

If $X \to S$ is log smooth of relative dimension $n$, the truncation maps $\ljsp{r}{X/S, m+1} \to \ljsp{r}{X/S, m}$
are $n$-dimensional affine bundles. 

\end{proposition}

\begin{proposition}[Relative Bhatt's theorem]\label{prop:relativebhatt}
If $X \to S$ is a morphism of fs log schemes, the relative arc space is the limit of the relative jet spaces 
\[\ljsp{r}{X/S, \infty} = \lim_m \ljsp{r}{X/S, m}.\]
\end{proposition}

\begin{example}[Thanks to Tommaso de Fernex]\label{ex:nodalcurve}

Let $Y$ be the affine log scheme
\[Y \coloneqq \Spec k[u, v, t]/(uv - t)\]
endowed with the log structure associated to the monoid
\[\NN\num{U, V, T}/(U + V = T)\]
generated by $U, V, T$, with the map sending $U \mapsto u$, etc. The map $Y \to \Aff^1$ projecting onto $t$ is log smooth. This is the prototypical node -- all log smooth curves $C \to S$ over a strictly henselian base $S$ are \'etale locally either smooth or pulled back from $Y \to \Aff^1$ \cite[pg. 222]{fkatologdefthy}. 

Because $Y \to S$ is log smooth, the truncation maps for the relative log jet spaces $\ljsp{r}{Y/S, m}$ are $\Aff^1$-bundles. This is true even for the fiber $Y_0 \to 0$ over the origin $0 \in \Aff^1$. This differs drastically from both the absolute log jets $\ljsp{r}{Y, m}$ and the ordinary jet space $\jsp{Y^\circ, m}$. 

Write $X = Y_0$ for the fiber over the origin, $X = \Spec k[u, v]/uv$. The ordinary jet space $\jsp{X^\circ, m}$ of the underyling scheme $X^\circ$ splits into $m+1$ irreducible components, one for each partition $m = i + j$. The $(i, j)$-component corresponds to $i$-jets along the component $u=0$ and $j$-jets along the $v=0$ component. Writing $u_0, u_1, \cdots, v_0, v_1, \cdots$ for Hasse-Schmidt derivatives of $u_0 = u, v_0 = v$ \cite{vojtajetspaces}, the $(i, j)$-component is cut out by the equations
\[u_0 = u_1 = \cdots = u_{i-1} = v_0 = \cdots v_{j-1}.\]
These are intersections of coordinate hyperplanes in the jet space $\jsp{\Aff^2, m}$. 

The strict inclusion $X \subseteq \Aff^2$ as the axes of $\Aff^2$ describes the absolute log jets as a pullback
\[\ljsp{r}{X, m} = \jsp{X, m} \times_{\jsp{\Aff^2, m}} \ljsp{r}{\Aff^2, m}. \]

Any family of nodal curves $C \to S$ can be canonically made into a log smooth family via placing this log structure at the nodes. Then this example describes the relative log jets $\ljsp{r}{C/S, m}$ strict-\'etale locally around the nodes and the case of $\Aff^1$ describes the smooth points. 

\end{example}

\begin{remark}

The ordinary relative jet space of the map $X \to \Log$ is another viable candidate for ``log jet space.'' Many notions in log geometry are simply relative notions for the map $X \to \Log$ or $X \to \af[X]$. This is a component of our notion $\ljsp{r}{X, m}$. 

The square 
\[
\begin{tikzcd}
\ljsp{r}{X, m} \ar[r] \ar[d] \lpbstrict      &\ev X \ar[d]      \\
\jsp{X^\circ, m} \ar[r]        &\jsp{\af[X]^\circ, m}
\end{tikzcd}
\]
of Remark \ref{rmk:logstackpbsquare} and the factorization of the trivial jets $\af[X] \subseteq \ev \af[X] \to \jsp{\af[x]^\circ, m}$ through the evaluation stack give a pullback square defining the relative log jets of $X \to \af[X]$
\[
\begin{tikzcd}
\jsp{X/\af[X], m} \ar[d] \ar[r] \lpbstrict &\ljsp{r}{X, m} \ar[r] \ar[d] \lpbstrict         &\jsp{X^\circ, m} \ar[d]       \\
\af[X] \ar[r]      &\ev \af[X] \ar[r]         &\jsp{\af[X]^\circ, m}.
\end{tikzcd}
\]
We see that $\jsp{X/\af[x], m} \to \ljsp{r}{X, m}$ is a union of components. The same goes for $X \to \Log$. 

\end{remark}

\section{The log tangent space of log jet spaces}\label{s:logtansplogjetsp}

\subsection{Ordinary jet spaces $m \neq \infty$}

We give another proof of the formula of \cite{arcspacedifflsdefernexdocampo} for the K\"ahler differentials of jet spaces $\jsp{X, m}$ for $m \neq \infty$, generalizing it to Weil restrictions.

\begin{situation}\label{sit:finitefree}
Let $S' \to S$ be a Gorenstein, finite, locally free morphism of constant degree $m \geq 1$ with $S$ locally noetherian. Consider $Y \to S'$ and suppose either that $S' \to S$ is the universal jet $\Spec \ZZ\adj{t}_m \to \Spec \ZZ$ or that $Y \to S'$ is quasiprojective. 
\end{situation}

Here $S$ is an arbitrary scheme over $\Spec \ZZ$; we need not work over $k$. The hypothesis of $X$ quasiprojective over $S$ is only necessary for representability in general. It is not necessary for jet spaces or for our results to apply. 

We freely omit various subscripts without risk of confusion. Our main example is $S' = \Spec k \adj{t}_m$, $S = \Spec k$. The reader may restrict to this example. In Situation \ref{sit:finitefree}, the Weil restriction 
\begin{equation}\label{eqn:weilrestn}
\WR_{S'/S} Y'(T) \coloneqq 
\left\{
\begin{tikzcd}
T \times_S S' \ar[rr, dashed] \ar[dr] & &Y \ar[dl]      \\
    &S'
\end{tikzcd}
\right\}    
\end{equation}
is representable by a scheme over $S$. This is the sheaf-theoretic pushforward of $Y \to S'$ along $S' \to S$. 

The map $\pi$ also admits a dualizing object \cite[0FKW]{sta}
\[\omega = \omega_\pi = \HHom_{\OO_S-mod}(\OO_{S'}, \OO_S).\]
The Gorenstein hypothesis ensures $\omega$ is an invertible sheaf \cite[0C08]{sta}. If $S', S$ are smooth for example, the map $\pi$ is a local complete intersection and hence Gorenstein. For jet spaces, $\omega$ is globally trivial. 

For most Weil restrictions of interest, the target $Y \to S'$ is pulled back from some $X \to S$.

\begin{situation}\label{sit:weilrestnmonogencase}

Suppose $X \to S$ is quasiprojective or $S' \to S$ is the universal jet. Define $Y \to S'$ as the pullback of $X \to S$ to $S'$. 

\end{situation}

\begin{remark}

In Situation \ref{sit:weilrestnmonogencase}, the functor of points \eqref{eqn:weilrestn} of the Weil restriction can then be written as 
\begin{equation}\label{eqn:weilrestnmonogencase}
\WR_{S'/S} (X \times_S S') =
\left\{\begin{tikzcd}
T \times_S S' \ar[rr, dashed, "s"] \ar[dr]       &   &X \times_S T \ar[dl]      \\
    &T
\end{tikzcd}\right\}.    
\end{equation}
In this case, there is a canonical open subscheme $\Gen_{S'/S} X \subseteq \WR_{S'/S} X \times_S S'$ on which the map $s$ in \eqref{eqn:weilrestnmonogencase} is a closed immersion \cite{mymonogeneity1}. The K\"ahler differentials of this open subscheme are pulled back from the ambient Weil restriction.

\end{remark}

\begin{example}

Jet spaces arise from Situation \ref{sit:weilrestnmonogencase}, with $S' \to S$ the universal jet $\Spec \ZZ \adj{t}_m \to \Spec \ZZ$ and $Y = X \times_S S'$ for a target $X$. For $m = 1$, $\WR_{S'/S} X \times_S S' = \jsp{X, 1} = T_X$ is the tangent space and the open subscheme $\Gen {X} \subseteq T_X$ is the complement of the zero section. 

\end{example}

Here are some more examples of Situations \ref{sit:finitefree} and \ref{sit:weilrestnmonogencase}:

\begin{example}

Let $S' = \Spec \CC^m$ and $S = \CC$, with the diagonal map $\CC \to \CC^m$ betwen them. The space $\WR_{S'/S} X \times_S S' = X^n$ parameterizes $m$ points in $X$. The subscheme $\Gen_{S'/S} X \times_S S' = {\rm Conf}_m(X)$ is the configuration space of points in $X$ where the points are not allowed to collide. 

\end{example}

\begin{example}

Let $S' = \OO_L$ and $S = \OO_K$ be an extension of number fields and take $X = \Aff^1$. Then $\WR_{S'/S} \Aff^1_{S'}$ evaluated on an $\OO_K$-algebra $C$ yields choices of elements $\theta \in C \otimes_{\OO_K} \OO_L$. The subscheme $\Gen_{S'/S} \Aff^1$ restricts to those elements $\theta$ which generate $C \otimes_{\OO_K} \OO_L$ as a $C$-algebra. Global sections over $\OO_K$ are \emph{monogenerators} for $\OO_L/\OO_K$, i.e. elements which generate $\OO_L$ as a ring over $\OO_K$. The classical \emph{Hasse problem} in number theory asks whether $S'/S$ admits monogenerators and how to determine them. 

\end{example}

We prove a warm-up lemma concerning functoriality of the Weil restriction.

\begin{lemma}

If $f : X \to Y$ is a monomorphism or closed immersion of $S'$-schemes, so is the induced map $f_* : \WR X \to \WR Y$ on Weil restrictions.

\end{lemma}

\begin{proof}

The case of monomorphisms results from the functor of points. 

For closed immersions, the claim is \'etale local in $Y$ and we can assume $Y = \Spec A$ and $X = \Spec B$ and the map $A \to B$ is surjective. We can also localize in $S$ to assume $S' \to S$ comes from a finite free algebra $C \to D$.

We need to show $\WR X \to \WR Y$ is proper. Let $R \subseteq K$ be a discrete valuation ring in its fraction field. We need to find a unique dashed lift
\[\begin{tikzcd}
\Spec K \times_S S' \ar[r] \ar[d]         &X  \ar[d]    &       &K \otimes_C D      &B \ar[l] \ar[dl, dashed]      \\
\Spec R \times_S S' \ar[r] \ar[ur, dashed]         &Y       &       &R \otimes_C D \ar[u, hook]      &A. \ar[u] \ar[l]
\end{tikzcd}\]
The map $R \otimes_C D \to K \otimes_C D$ remains injective by the flatness of $C \to D$. The surjection $A \to B$ can detect whether elements of $B$ map into the subring $R \otimes_C D$, so we are done.

\end{proof}

\begin{example}

If $f : X \to Y$ is proper, $f_* : \WR {X} \to \WR {Y}$ need not be. Even for $\PP^n \to \pt$, the jet spaces of $\PP^n$ are affine bundles over $\PP^n$ and thus not proper unless $m = 0$. 

\end{example}

We need to understand what the Weil restriction does to vector bundles. If $F$ is a sheaf of modules on $S'$, we can form 
\[\VV(F) = \Spec \Sym F.\]
The Weil restriction of $\VV(F)$ does \emph{not} coincide with the vector bundle associated to the pushforward $\pi_* F$! Here is the correct formula.

\begin{lemma}\label{lem:weilrestnvb}

The Weil restriction of a scheme $\VV(F)$ along $S' \to S$ is computed using the dualizing sheaf $\omega$:
\begin{equation}\label{eqn:weilrestnvb}
\WR_{S'/S} \VV(F) = \VV(\pi_*(F \otimes \omega)).    
\end{equation}
\end{lemma}

\begin{proof}

Check this identity on functors of points.
\begin{equation}\label{eqn:pfwdadjoint}
\begin{split}
\left(\WR_{S'/S} \VV(F)\right)(T) &= \VV(F)(T\times_S S')         \\
        &= \Hom(F, \OO_T|_{S'})       \\
        &\overset{\otimes \omega}{\longsimeq} \Hom(F \otimes \omega, \pi^! \OO_T)       \\
        &=\Hom(\pi_* (F \otimes \omega), \OO_T)         \\
        &=\VV(\pi_*(F \otimes \omega))(T).
\end{split}    
\end{equation}
See \cite[0BUL]{sta} for the connection of $\omega$ to duality $\pi^!$ used here. 

\end{proof}

The pullback $\cal U$ of $S' \to S$ along $\WR X \to S$ has a universal map to $X$
\[\begin{tikzcd}
\cal U \ar[d, "\rho", swap] \ar[r, "f"]     &X \ar[d]      \\
\WR {X} \ar[r]     &S.
\end{tikzcd}\]

\begin{lemma}\label{lem:jetsptansp}
The relative tangent bundle $T_{\WR {X}/S}$ to the Weil restriction is the Weil restriction of the tangent bundle $T_{X/S}|_\cal U$ along $\rho$
\[T_{\WR_{X}/S} \simeq \WR_{\cal U/\WR X} \left(T_{X/S}|_{\cal U}\right).\]
The $k$th relative jet space of the $m$th jet space satisfies the same relation
\[\jsp{\WR {X}/S, k} \simeq \WR_{\cal U/\WR_{X}} \left(\jsp{X/S, k}|_{\cal U}\right).\]
\end{lemma}

The latter formula more generally holds for compositions of distinct Weil restrictions. 

\begin{proof}

The tangent bundle is the case $k = 1$. Unpack the functor of points on an $S$-scheme $T$, writing $T' = T \times_S S' = T \times_{\WR X} \cal U$
\[
\jsp{\WR X/S, m}(T) \coloneqq 
\left\{
\begin{tikzcd}
T \ar[r] \ar[d]       &\WR X \ar[d]      \\
T \times \jet{m} \ar[r] \ar[ur, dashed]        &S
\end{tikzcd}    
\right\}.
\]
Rewrite this diagram. 
\[
\begin{tikzcd}
T' \ar[r] \ar[d]      &\cal U \ar[r]  &X \ar[d]      \\
T' \times \jet{m} \ar[r] \ar[urr, dashed]   &T \ar[r]      &S.
\end{tikzcd}
\]
But these are sections of $\jsp{X/S}|_{\cal U}$ over $T'$, or sections of the Weil restriction $\WR_{S'/S} \jsp{X/S}|_{\cal U}$ over $T$. 

\end{proof}

If $S'/S$ is the universal jet $\Spec k \adj{t}_m/\Spec k$, then $T \times_{\jsp{X, m}} \cal U = T \times \jet{m}$. The final diagram in the proof is then 
\[\begin{tikzcd}
T \times \jet{m} \ar[r] \ar[d]        &\cal U \ar[r]       &X      \\
T \times \jet{m} \times \jet{k} \ar[urr, dashed, bend right]
\end{tikzcd}.\]

\begin{corollary}[{\cite[Theorem B]{arcspacedifflsdefernexdocampo}}]\label{cor:jetspdiffls}
The K\"ahler differentials of the Weil restriction $\WR {X}$ over $S$ are
\[\kah{\WR {X}/S} \simeq \rho_*(\kah{X/S}|_\cal U \otimes \omega).\]
\end{corollary}

\begin{proof}

Write Lemma \ref{lem:jetsptansp} using Lemma \eqref{lem:weilrestnvb}. 

\end{proof}

This computes the differentials of the jet space as well as the configuration space ${\rm Conf}_m(\Aff^1)$ of points in $\CC$ and the scheme of monogenerators $\Gen_{\OO_L/\OO_K} \Aff^1$ in one fell swoop. For $X = \Aff^1$, each space is smooth.

\begin{remark}

See the forthcoming work of C. Eric Overton-Walker for a version of Corollary \ref{cor:jetspdiffls} in derived geometry. 

A sibling of Lemma \ref{lem:weilrestnvb} appears in the proof of \cite[Proposition 19.1.1.2]{luriesag}. Likewise Corollary \ref{cor:jetspdiffls} is similar to Proposition 19.1.4.1 of loc. cit. 

\end{remark}

\subsection{Ordinary arc spaces $m = \infty$}

The log version of Bhatt's theorem in Proposition \ref{prop:bhattthm} writes the arc space $\ljsp{r}{X, \infty}$ as a limit of jet spaces under the truncation maps. We can understand the tangent space of $\ljsp{r}{X, \infty}$ as a limit. Write 
\[\omega_\infty = \colim_m \omega_m|_{\ljet{r}{\infty}}\]
for the colimit of the dualizing sheaves of $\Spec \ZZ \adj{t}_m \to \Spec \ZZ$. Each $\omega_m$ can be trivialized over $\ljet{r}{m}$, but the colimit is different.

The sheaf $\omega_m = \HHom_{\ZZ-mod}(\ZZ \adj{t}_m, \ZZ)$ is canonically isomorphic to the module 
\begin{equation}\label{eqn:omegajets}
\omega_m \simeq t^{-m} \ZZ[t]/t \ZZ[t]    
\end{equation}
as in \cite[Lemma 4.2]{arcspacedifflsdefernexdocampo}. The maps truncation maps of jets induce maps $\omega_m|_{\jet{m+1}} \to \omega_{m+1}$. Under the identification \eqref{eqn:omegajets}, the maps between the $\omega$'s extend a sum $\sum_{j = 0}^m a_j t^{-j}$ by taking $a_{-m-1} = 0$. 

\begin{lemma}

The colimit $\omega_\infty = \colim_m \omega_m|_{\ljet{r}{\infty}}$ is 
\[\omega_\infty \simeq \ZZ((t))/t \ZZ \adj{t},\]
as a module on $\ljet{r}{\infty}$. The underlying abelian group is
\[\omega_\infty \overset{\ZZ-mod}{\simeq} \bigoplus_{j = 0}^\infty \ZZ t^{-j}.\]

\end{lemma}

\begin{proof}

Omitted. 

\end{proof}

\begin{proposition}

The log K\"ahler differentials of the log arc space $\Tl{\ljsp{r}{X/S, \infty}/S}$ are the limit of the formula from Corollary \ref{cor:jetspdiffls}
\[\lkah{\ljsp{r}{X/S, \infty}/S} \simeq \colim_m \rho_*\left(\lkah{X/S}|_{\cal U} \otimes \omega_m \right) \simeq \rho_*\left(\lkah{X/S}|_{\cal U} \otimes \omega \right).\]

\end{proposition}

\begin{proof}

The relative version of Bhatt's theorem $\ljsp{r}{X/S, \infty} = \lim_m \ljsp{r}{X/S, m}$ from Proposition \ref{prop:relativebhatt} formally implies the log tangent space of the log arc space $\Tl{\ljsp{r}{X/S, \infty}/S}$ is the limit of the log tangent spaces of the log jets. Likewise the log K\"ahler differentials are obtained by a colimit. 

\end{proof}

\subsection{The log tangent space $\Tl{X}$ of log points}

We remind the reader about the log tangent space before proving log analogues of the above results. 

\begin{definition}

The {\em log tangent space} of a log scheme $X$ is the scheme
\[\Tl{X} \coloneqq \VV(\lkah{X}) = \SSpec{X} \Sym \lkah{X}\]
over $X$ associated to the sheaf of log K\"ahler differentials. If $X \to Y$ is a morphism of log schemes, one can relativize this construction 
\[\Tl{X/Y} \coloneqq \VV(\lkah{X/Y}).\]

\end{definition}

If $X \to Y$ is log smooth, $\lkah{X/Y}$ is locally free and $\Tl{X/Y}$ is a vector bundle.

\begin{remark}\label{rmk:logtanspfunctorofpoints}

We claim the functor of points for $\Tl{X/Y}$ on $X$-schemes $S$ consists of diagrams 
\begin{equation}\label{eqn:logtansp}
\begin{tikzcd}
S \ar[r, "g"] \ar[d]      &X \ar[d]      \\
S \times \jet{1} \ar[r] \ar[ur, dashed]        &Y
\end{tikzcd},    
\end{equation}
where $\jet{1}$ is given the trivial log structure. Since $S \subseteq S \times \jet{1}$ is a log thickening, \cite[Theorem IV.2.2.2]{ogusloggeom} shows the sheaf of dashed arrows form a pseudotorsor under log derivations
\[\begin{split}
\Der^\ell_{X/Y}(g_* \OO_S) &= \Hom(\lkah{X/Y}, g_* \OO_S) \\
&= \Hom(g^* \lkah{X/Y}, \OO_S) \\
&= \Tl{X/Y}(S)    
\end{split}.\]
The retraction $S \times \jet{1} \to S$ confirms that the sheaf of diagrams \eqref{eqn:logtansp} is a torsor under $\Tl{X/Y}(S)$ and also trivializes the torsor. Thus $\Tl{X/Y}(S)$ is the set of such diagrams \eqref{eqn:logtansp}.

\end{remark}

These are not to be confused with the tangent sheaf $\cal T_X = (\kah{X})^\vee$, the dual of the sheaf of K\"ahler differentials that we do not use in this article.

If $X = X^\circ$ has trivial log structure or more generally if the map $X \to Y$ is strict, the log K\"ahler differentials and log tangent space coincide with the ordinary ones
\[\lkah{X/Y} = \kah{X^\circ/Y^\circ}, \qquad \Tl{X/Y} = T_{X^\circ/Y^\circ} = \VV(\kah{X^\circ/Y^\circ}).\]

One surprise with the log tangent space is that points with log structure can have higher dimensional log tangent spaces. 

\begin{example}[Standard log point]

Consider the standard log point $P_{\NN}$ of Definition \ref{defn:evalsp}. The log tangent space $\Tl{P_{\NN}}$ is one-dimensional, whereas the ordinary tangent space $T_{\Spec k}$ of $P_\NN^\circ = \Spec k$ is zero-dimensional. 

\end{example}

The log point $P_{\NN^k}$ with rank $k$ log structure has tangent space of dimension $k$. We now generalize this for later use.

\begin{remark}\label{rmk:leslogcotcplx}
Let $X \to Y \to Z$ be a sequence of fs log schemes with $Y \to Z$ log smooth. Obtain an exact sequence 
\[0 \to h^{-1} \lccx{X/Z} \to h^{-1} \lccx{X/Y} \to \lkah{Y/Z}|_X \to \lkah{X/Z} \to \lkah{X/Y} \to 0\]
using O. Gabber's log cotangent complex and then identifying the low degree terms with the log cotangent complex $\lccx{-}$ of M. Olsson \cite[Theorem 8.27]{logcotangent}. 

If $X \to Y$ is strict, its log cotangent complex coincides with the ordinary one $\lccx{X/Y} = \ccx{X/Y}$. If $X \to Y$ is a strict closed embedding with ideal $I$, the first homology group $h^{-1} \lccx{X/Y}$ of the log cotangent complex is $I/I^2$. 

\end{remark}

Let $\alpha = \Spec L$ be a point of finite type over $\Spec k$ with log structure $M_\alpha = L^* \oplus Q$ for some fs sharp monoid $Q$. There is a strict embedding of $\alpha$ as the vertex inside the affine toric variety $\Aff_{Q, L} = \Spec L[Q]$. The vertex map $L[Q] \to L$ sends the nonunits $Q^+ \subseteq Q$ to $0 \in L$ and $0 \in Q$ to $1 \in L$. Its kernel is $J = Q^+ L[Q]$. The quotient $J/J^2$ is generated by irreducible elements of $Q^+$ as an $L$ vector space, so it has dimension 
\[\dim_L J/J^2 = \#(Q^+ \setminus Q^{+2}).\]

Since the toric variety $\Aff_{Q, L}$ is log smooth over $\alpha^\circ$, there is an exact sequence
\begin{equation}\label{eqn:logcotangentvertex}
0 \to h^{-1} \lccx{\alpha/\alpha^\circ} \to J/J^2 \overset{dlog}{\longrightarrow} \lkah{\Aff_{Q, L}/\alpha^\circ}|_{\alpha} \to \lkah{\alpha/\alpha^\circ} \to 0.   
\end{equation}
We claim the map $dlog$ is zero.

\begin{proposition}\label{prop:logcotcplxlogpoint}

The map $dlog$ in the sequence \eqref{eqn:logcotangentvertex} is zero. We have identifications
\[h^{-1} \lccx{\alpha/\alpha^\circ} \simeq J/J^2\]
\[\lkah{\Aff_{Q, L}/\alpha^\circ}|_{\alpha} \simeq \lkah{\alpha/\alpha^\circ}.\]
The dimensions are then
\[\dim_L h^{-1} \lccx{\alpha/\alpha^\circ} = \#(Q^+ \setminus Q^{+2})\]
\[\dim \lkah{\alpha/\alpha^\circ} = \rk \gp Q.\]

\end{proposition}

\begin{proof}

Since \eqref{eqn:logcotangentvertex} is a complex of $L$-vector spaces, we can count dimensions to see that $dlog$ is zero. It suffices to check that
\[\dim \lkah{\alpha/\alpha^\circ} = \rk \gp Q = \dim \lkah{\Aff_{Q, L}/\alpha^\circ}|_{\alpha}.\]
The first equality holds by \cite[Corollary IV.2.3.6]{ogusloggeom} and the latter because $\Aff_{Q, L} \to \alpha^\circ$ is log smooth of dimension $\rk \gp Q$.
    
\end{proof}

If the ground field $k$ is perfect, any finite type field extension $\alpha^\circ/k$ is smooth and $\dim \kah{\alpha^\circ/k}$ is the transcendence degree of $L/k$. The sequence
\begin{equation}\label{eqn:absvsrelativepoint}
0 \to h^{-1} \lccx{\alpha} \to h^{-1} \lccx{\alpha/\alpha^\circ} \to \kah{\alpha^\circ} \to \lkah{\alpha} \to \lkah{\alpha/\alpha^\circ} \to 0  
\end{equation}
is exact by Remark \ref{rmk:leslogcotcplx}. 

This exact sequence \eqref{eqn:absvsrelativepoint} relates the Euler characteristic of the 2-truncated absolute cotangent complex of $\alpha$ with the relative one. 

\begin{lemma}\label{lem:eulerchar2termlogcot}
Let $k$ be perfect. The alternating sums of the terms in the absolute and relative log cotangent complexes differ by the dimension of $\alpha^\circ$ over $k$
\[\begin{split}
\dim \lkah{\alpha} - \dim h^{-1} \lccx{\alpha} &= \dim \lkah{\alpha/\alpha^\circ} - \dim h^{-1} \lccx{\alpha/\alpha^\circ} + \dim_k {\alpha^\circ}         \\
&=\rk \gp Q - \#(Q^+ \setminus Q^{+2}) + \dim_k \alpha^\circ.
\end{split}\]

\end{lemma}

\subsection{The log tangent space of $\ljsp{r}{X, m}$}

We define the log Weil restriction $\WR^\ell$ similarly, before proving it coincides with the ordinary Weil restriction for vector bundles $\VV(F)$. Let $S' \to S$ be an integral and saturated, Gorenstein, finite, locally free morphism of log schemes and $Y \to S'$ quasiprojective. Write $\WR^\ell_{S'/S} Y$ for the functor on log schemes $T$ over $S$
\begin{equation}\label{eqn:logweilrestn}
\WR^\ell_{S'/S} Y(T) \coloneqq 
\left\{
\begin{tikzcd}
T \times_S^\ell S' \ar[rr, dashed] \ar[dr]      &       &Y \ar[dl]      \\
        &S'
\end{tikzcd}
\right\}.    
\end{equation}
Since $S' \to S$ is integral and saturated, the fs fiber product $T \times^\ell_S S'$ coincides with the scheme theoretic fiber product on underlying schemes. The only difference from the ordinary Weil restriction is that all the morphisms in \eqref{eqn:logweilrestn} are of log schemes. Continue to write $\omega$ for the ordinary dualizing sheaf of the map $S' \to S$. 

Given a coherent sheaf $F$ on a log scheme $Y$, endow $\VV(F)$ with the log structure pulled back from $Y$.

\begin{lemma}

The ordinary Weil restriction of $\VV(F)$ coincides with the log Weil restriction on underlying schemes 
\[(\WR^\ell_{S'/S} \VV(F))^\circ = \WR_{S'/S} \VV(F).\]
The log structure of $\WR^\ell_{S'/S} \VV(F)$ is pulled back from the base $S$. 

\end{lemma}

\begin{proof}

Omitted. 

\end{proof}

Write $\cal U$ for the universal object $\WR^\ell_{S'/S} Y \times_S S'$ over the log Weil restriction $\WR^\ell_{S'/S} Y$
\[\begin{tikzcd}
\cal U \ar[r] \ar[d, "\rho"] \lpbstrict      &S' \ar[d]         \\
\WR^\ell Y \ar[r]         &S.
\end{tikzcd}\]
Obtain the log versions the same way.

\begin{proposition}\label{prop:logtansplogjetsp}

The log tangent space of the log Weil restriction of $Y$ is the Weil restriction of the log tangent space 
\[\Tl{\WR^\ell Y/S} = \WR^\ell_{S'/S} \Tl{Y/S}|_{\cal U}.\]
\[\ljsp{r}{\WR^\ell Y/S, m} = \WR^\ell_{S'/S} \ljsp{r}{Y/S, m}|_{\cal U}\]
\[\lkah{\ljsp{r}{Y, m}/S} = \rho_*\left(\lkah{Y/S}|_\cal U \otimes \omega\right).\]

\end{proposition}

\section{The embedding dimension of the log jet space}\label{s:fittingideals}

We recall some notions from commutative algebra before our main results. This section closely parallels \cite[\S 6]{arcspacedifflsdefernexdocampo}. 

We state our main application, computing the embedding dimension of points of the log jet space. The rest of the section explains the terminology and the meaning of the formula before the proof. 

\begin{theorem}[{\cite[Lemma 8.1]{arcspacedifflsdefernexdocampo}}]\label{thm:embdim}

Suppose $X$ is finite type over a perfect ground field $\Spec k^\circ$ with trivial log structure.  Let $\alpha_m \in |\ljsp{r}{X, m}|$ be a point in the sense of \cite[03BT]{sta} which lifts to an arc $\alpha_\infty \in |\ljsp{r}{X, \infty}|$. 

The embedding dimension $\embdim\left(\OO_{\ljsp{r}{X, m}, \alpha_m}\right)$ of the log jet space at $\alpha_m$ is bounded by
\begin{equation}\label{eqn:embdimthm}
\begin{split}
\embdim\left(\OO_{\ljsp{r}{X, m}, \alpha_m}\right) \leq &d_m(m+1) + \ord_{\alpha_\infty} \left(\Fitt^{d_m}(\lkah{X}|_{\alpha_\infty \times \ljet{r}{\infty}})\right) \\
&- \rk \gp{\bar M}_{\alpha_m} + N - \dim_k \alpha^\circ.
\end{split}
\end{equation}

We have equality if $X$ is log smooth. The number $N$ is the number of irreducible elements in the stalk of the sheaf $\bar M_\alpha$ at a geometric point and $d_m$ is the Betti number (Definition \ref{defn:bettinum}) of $\lkah{X}|_{\alpha_m \times \ljet{r}{m}}$.

\end{theorem}

Remark \ref{rmk:relativeembdim} addresses the relative setting $X/S$.

\subsection{Commutative algebra preliminaries}\label{ss:commalg}

This section is purely expository. 

The embedding dimension is usually defined only for the closed point of a local ring. We extend this definition to points $\alpha \in |Y|$ of the topological space associated to an algebraic space $Y$.

\begin{definition}\label{defn:embdim}

Let $(B, m)$ be a local ring. Recall the \textit{embedding dimension} of $(B, m)$ is the dimension of the Zariski tangent space at the closed point 
\[\embdim\left(\OO_{B}\right) = \dim_{B/m} m/m^2.\]
The embedding dimension of a scheme $Y$ at a point $\alpha \in Y$ is that of the local ring $\OO_{Y, \alpha}$. 

This number does not change under separable extensions of the field $L$ or factorizations $\alpha \to \Spec B' \to \Spec B$ with $B \to B'$ \'etale. Suppose $Y$ is an algebraic space and $\alpha \in |Y|$ is a point of $Y$ in the sense of \cite[03BT]{sta}. There is an \'etale cover $W \to Y$ by a scheme $W$ and an honest point $\beta \in W$ mapping to $\alpha$. The embedding dimension of $(W, \beta)$ is independent of choices, so we define it to be the embedding dimension of $Y$ at $\alpha$. 

\end{definition}

The Zariski tangent space of $B$ at $\alpha$ can be thought of as the normal sheaf $N_{\alpha/\Spec B}$ of Behrend and Fantechi \cite{intrinsic}. Our arguments can be rewritten in that language.

We need to compare points $\alpha_m \to \ljsp{r}{X, m}$ for different orders $m$ of jets, in particular with arcs $\alpha_\infty \to \ljsp{r}{X, \infty}$.

\begin{definition}

Let $n \geq m$ be a positive integer or infinity $n \in \NN \cup \{\infty\}$. A \emph{lift} of a point $\alpha_m \to \ljsp{r}{X, m}$ \emph{to order $n$} is a field $\alpha_n = \Spec L_n$ and a commutative diagram 
\begin{equation}\label{eqn:arcliftingjet}
\begin{tikzcd}
\alpha_n \times \ljet{r}{m} \ar[r] \ar[d]         &\alpha_n \times \ljet{r}{n} \ar[d]   \\
\alpha_m \times \ljet{r}{m} \ar[r]         &X.
\end{tikzcd}    
\end{equation}

\end{definition}

A lift to order $n$ is equivalent to a diagram 
\[
\begin{tikzcd}
\alpha_n \ar[r] \ar[d]        &\ljsp{r}{X, n} \ar[d]  \\
\alpha_m \ar[r]        &\ljsp{r}{X, m}.
\end{tikzcd}
\]
Such a lift induces lifts to all orders $k$ in between $m \leq k \leq n$.

A coherent sheaf $M$ on $\alpha_m \times \ljet{r}{m}$ is a finitely generated $L_m \adj{t}_m$-module. We describe the dimension of $M$ as an $L_m$ vector space. This dimension does not change under base change $-\otimes_{L_m} L_n$ of fields along a lift of $\alpha_m$ to order $n$, but it will change under varying the order $-\otimes_{L_m \adj{t}_m} L_n \adj{t}_n$.

Since $M$ is also a module under $L_m[t]$, the fundamental theorem of finitely generated modules over a principal ideal domain supplies a decomposition
\[M \simeq \prod_i L_m[t]/t^{e_i}.\]
The numbers $e_i$ range up to $m+1$, since $M$ is an $A_m$-module. Assume they are in decreasing order $e_{i+1} \geq e_i$. Write $d_m$ for the index of the first $e_i$ that is not equal to $m+1$, so that
\begin{equation}\label{eqn:ftfgag}
M \simeq A_m^{d_m} \times \prod_i A_m/t^{e_i}.    
\end{equation}

\begin{definition}\label{defn:bettinum}
The integer $d_m$ is called the \emph{Betti number} of $M$. 
\end{definition}

The presentation \eqref{eqn:ftfgag} exhibits $M$ as the cokernel of a diagonal matrix $A_m^{N_1} \to A_m^{N_2}$, with entries $t^{e_0}, t^{e_1}, \dots$ along the diagonal. 

The $i$th Fitting ideal $\Fitt^i M$ of the module $M$ is generated by the determinants of the $k$-minors of a presentation for $M$. The number $k$ satisfies $k + i = N_2$. The determinants of the minors are the products
\[t^{e_{i_1}} \cdots t^{e_{i_k}} = t^{\sum e_{i_j}}.\]
The generator for the ideal is the smallest such power, with exponent $E_i = \sum_{j \geq i} e_j$. This sum $E_i$ is the dimension of the remainder over $L_m$. If $E_i \geq m+1$, $t^{E_i} = 0 \in A_m$ and we have the zero ideal $(0) = (t^{m+1})$. 

The order of vanishing of the Fitting ideal is precisely this exponent, so 
\[\ord_{\alpha_m} \Fitt^i M = \min(m+1, E_i).\]
If $\alpha_m$ is an arc with $m = \infty$, the order of vanishing is exactly the sum $E_i$. 

Now suppose $\alpha_m$ has a lift to an arc $\alpha_\infty$, so we have specified lifts to each order $n \geq m$. Use analogous notation $d_n$, $e_{n, i}$, etc. Write $M_m = M$ and 
\[M_n = M \otimes_{L_m \adj{t}_m} L_n \adj{t}_n.\]

If $M$ is the pullback of a module $Q$ on $X$ to $\alpha_m \times \ljet{r}{m}$, the dimension $E_i$ of the remainder can be computed on the arc $\alpha_\infty$
\[E_i = E_{m, i} = \ord_{\alpha_\infty} \Fitt^{d_m} \left(Q|_{\alpha_\infty \times \ljet{r}{\infty}}\right).\]
One must know the $m$th Betti number $d_m$.

As $n$ increases, more of the free part $A_n^{d_n}$ in \eqref{eqn:ftfgag} becomes part of the remainder $\prod_i A_n/t^{e_i}$
\[d_m \geq d_{m+1} \geq \cdots \geq d_\infty.\]
Likewise, the sum $E_{n, i} = \sum_{j \geq i} e_{n, j}$ does not decrease as $n$ increases
\[E_{m, i} \leq E_{m+1, i} \leq \cdots \leq E_{\infty, i}.\]
These numbers stabilize as soon as $m > e_i$ for the largest torsion order $e_i = e_{d_\infty}$. In particular, it suffices to have $m > E_\infty$.

\subsection{Embedding dimension of log jet spaces}

Consider the pullback of the universal jet $\cal U = \ljsp{r}{X, m} \times \ljet{r}{m}$. 
\[\begin{tikzcd}
\alpha_m \times \ljet{r}{m} \ar[r] \ar[d, "\sigma"] \lpbstrict         &\cal U \ar[r, "f"] \ar[d, "\rho"]      &X \ar[d]    \\
\alpha_m \ar[r]        &\ljsp{r}{X, m} \ar[r]   &S
\end{tikzcd}\]
Rewrite the stalk $\lkah{\ljsp{r}{X, m}}$ at $\alpha_m$ using the formula for the log K\"ahler differentials of Proposition \ref{prop:logtansplogjetsp}
\[\lkah{\ljsp{r}{X, m}}|_{\alpha_m} = \rho_*(\lkah{X}|_{\cal U} \otimes \omega)|_{\alpha_m} = \sigma_*(\lkah{X}|_{\alpha_m \times \ljet{r}{m}} \otimes \omega) \simeq \sigma_*(\lkah{X}|_{\alpha_m \times \ljet{r}{m}}).\]
We now compute the dimension of this $L_m$-vector space.

Apply Subsection \S \ref{ss:commalg} to the module $M = \lkah{X}|_{\alpha_m \times \ljet{r}{m}}$. Taking dimensions as $L_m$ vector spaces $\dim_{L_m}-$ of Equation \eqref{eqn:ftfgag} yields

\begin{corollary}

Suppose the jet $\alpha_m$ lifts to an arc $\alpha_\infty$. The dimension of the stalk of $\lkah{\ljsp{r}{X, m}}$ at ${\alpha_m}$ is
\begin{equation}\label{eqn:ordfittingjacobian}
\begin{split}
\dim_{L_m} \left(\lkah{\ljsp{r}{X, m}}|_{{\alpha_m}}\right) &= \dim_{L_m} \left(\sigma_* \lkah{X}|_{{\alpha_m} \times \ljet{r}{m}}\right) \\
&= d_m \cdot (m+1) + \ord_{{\alpha_\infty}} \left(\Fitt^{d_m}(\lkah{X}|_{\alpha_\infty \times \ljet{r}{\infty}})\right).   
\end{split}
\end{equation}
The number $d_m$ is the Betti number of $\lkah{X}|_{\alpha_m \times \ljet{r}{m}}$. 

\end{corollary}

Let $J \to \ljsp{r}{X, m}$ be a strict étale cover with $\alpha \in J$. We can freely replace $\alpha_m$ by finite separable field extensions. Write $D = \bar {\{\alpha_m\}} \subseteq J$ for the closure of the point $\alpha_m$ with the restriction of the log structure of $J$. Let $I$ be the ideal sheaf of $D \subseteq J$. 

We have a long exact sequence for the triple $D \to J \to \Spec k$, which we pull back to $\alpha_m \in D$
\begin{equation}\label{eqn:leslogcotcplx}
\cdots \to h^{-1} \lccx{{\alpha_m}} \to I/I^2 \to \lkah{\ljsp{r}{X, m}}|_{{\alpha_m}} \to \lkah{{\alpha_m}} \to 0.    
\end{equation}
as in Remark \ref{rmk:leslogcotcplx}. If $X$ is log smooth, $J$ is as well and the sequence does not continue to the left.

This bounds the embedding dimension $\dim_{L_m} I/I^2$ of $\alpha_m \in J$.
\begin{equation}\label{eqn:embdim1}
\embdim \left(\OO_{J, {\alpha_m}}\right) \leq \dim \lkah{J}|_{{\alpha_m}} - \dim \lkah{{\alpha_m}} + \dim h^{-1} \lccx{{\alpha_m}}    
\end{equation}

The terms in the sequence \eqref{eqn:leslogcotcplx} are all stable under finite separable field extension. We can then assume the log structure of ${\alpha_m}$ is the constant sheaf $M_{\alpha_m} = L_m^* \times Q$, with $\underline Q = \bar M_{\alpha_m}$.

\begin{proof}[Proof of Theorem \ref{thm:embdim}]

The number $N$ in the formula \eqref{eqn:embdimthm} is the number of irreducible elements $\#(Q^+ \setminus Q^{+2})$. Rewrite \eqref{eqn:embdim1} using Lemma \ref{lem:eulerchar2termlogcot}. 

\end{proof}

We can obtain this theorem for relative log jets $\ljsp{r}{X/S, m}$ under more restrictive hypotheses.

\begin{remark}[Embedding dimension of relative log jets]\label{rmk:relativeembdim}

Let $X \to S$ be log smooth and finite type. Suppose $\alpha_m \to S$ is either log smooth or that the underlying map of schemes $\alpha_m^\circ \to S^\circ$ is smooth. One can formulate a weaker relative version of the formula \eqref{eqn:embdimthm} for a point of the relative log jet space $\alpha_m \to \ljsp{r}{X/S, m}$. One can immediately relativize \eqref{eqn:ordfittingjacobian} and \eqref{eqn:leslogcotcplx}. The tricky part is Lemma \ref{lem:eulerchar2termlogcot}.

Write $\beta$ for the point ${\alpha_m}$ endowed with the log structure of $S$. By our assumptions on $\alpha_m \to S$, either it is log smooth or $\beta \to S$ is. In either case, the triple ${\alpha_m} \to \beta \to S$ leads to an exact sequence
\begin{equation}\label{eqn:rellogcotlogpoint}
0 \to h^{-1} \lccx{{\alpha_m}/S} \to h^{-1} \lccx{{\alpha_m}/\beta} \to \lkah{\beta/S} \to \lkah{{\alpha_m}/S} \to \lkah{{\alpha_m}/\beta} \to 0.    
\end{equation}
This provides a formula
\[\dim \lkah{{\alpha_m}/S} - \dim h^{-1} \lccx{{\alpha_m}/S} = \dim \lkah{{\alpha_m}/\beta} - \dim h^{-1} \lccx{{\alpha_m}/\beta} + \dim \lkah{\beta/S}.\]
Since $\beta/S$ is strict, $\lkah{\beta/S} = \kah{\beta^\circ/S^\circ} = \kah{\alpha_m^\circ/S^\circ}$.

To compute $\lkah{{\alpha_m}/\beta}$ and $h^{-1} \lccx{{\alpha_m}/\beta}$, consider the exact sequence for ${\alpha_m} \to \beta \to \beta^\circ = \alpha_m^\circ$.
\[
\cdots \to h^{-1} \lccx{\beta/\beta^\circ} \to h^{-1} \lccx{{\alpha_m}/\beta^\circ} \to h^{-1} \lccx{{\alpha_m}/\beta} \to \lkah{\beta/\beta^\circ} \to \lkah{{\alpha_m}/\beta^\circ} \to \lkah{{\alpha_m}/\beta} \to 0
\]
After separable field extension and applying the calculations in the absolute case, we get a long exact sequence
\[
\cdots \to L_m^{\#(R^+ \setminus R^{+2})} \to L_m^{\#(Q^+ \setminus Q^{+2})} \to h^{-1} \lccx{{\alpha_m}/\beta} \to L_m^{\rk \gp R} \to L_m^{\rk \gp Q} \to \lkah{{\alpha_m}/\beta} \to 0
\]
of $L_m$-vector spaces. Here $Q$ and $R$ are the characteristic monoids of ${\alpha_m}, \beta$ respectively. This shows the log K\"ahler differentials are the relative characteristic monoid $\lkah{{\alpha_m}/\beta} = \gp {\bar M_{{\alpha_m}/S}}$ and $h^{-1} \lccx{{\alpha_m}/\beta}$ is generated by the kernel of $\gp R \to \gp Q$ and the irreducible elements of $Q^+$ not in the image of irreducible elements of $R^+$.

Combining the above, one gets
\begin{equation}\label{eqn:relativeembdimthm}
\begin{split}
\embdim\left(\OO_{\ljsp{r}{X/S, m}, \alpha_m}\right) = &d_m(m+1) + \ord_{\alpha_\infty} \left(\Fitt^{d_m}(\lkah{X/S}|_{\alpha_\infty \times \ljet{r}{\infty}})\right)    \\ 
&- \rk \gp{\bar M}_{\alpha_m/S} + N - \dim \kah{\alpha_m^\circ/S^\circ}.
\end{split}
\end{equation}
Here $N$ is the sum of the number of irreducible elements of $Q^+$ not in the image of those of $R^+$ with the rank of the kernel of $\gp R \to \gp Q$.

One could demand weaker conditions than we have for $\alpha_m \to S$ at the cost of a longer sequence than \eqref{eqn:rellogcotlogpoint}, for example that ${\alpha_m} \to \beta$ is a ``log locally complete intersection.'' The above form suffices for cases where $S$ is a log point, such as the log smooth nodal curves $C \to S$ discussed in Example \ref{ex:nodalcurve}.

\end{remark}

The usual formula for the embedding dimension of a point ${\alpha_m} \in \jsp{X, m}$ that lifts to an arc is
\[\embdim\left(\OO_{\jsp{X, m}, {\alpha_m}}\right) = d_m(m+1) - \dim \alpha_m^\circ + \ord_{\alpha_\infty} \Fitt^{d_m} \kah{X}.\]
The log version \eqref{eqn:embdimthm} differs in a few ways. Log K\"ahler differentials $\lkah{X}$ are used and there are correction terms $N, \rk \gp {\bar M}_{\alpha_m}$ coming from $\lccx{\alpha/\alpha^\circ}$ by Lemma \ref{lem:eulerchar2termlogcot}. Furthermore, $X$ is required to be log smooth.

Formula \eqref{eqn:embdimthm} is as far as we could push the results of \cite{arcspacedifflsdefernexdocampo} in the log setting. We elaborate on the obstructions to obtaining their theorems and the differences between log jets and ordinary ones.

\begin{remark}

The equality of the embedding dimension with the jet codimension obtained in \cite[Theorem C]{arcspacedifflsdefernexdocampo} does not immediately carry over to log jet spaces. 

Assume $\alpha_n \to \alpha_m$ is a lift that is strict. The same argument as in Lemma 8.3 of loc. cit. in similar notation shows
\[\dim_{L_\infty} ({\rm Im} \lambda_{m, n}) \geq d (m+1) - \dim \alpha_m^\circ + N - \rk \gp{\bar M}_{\alpha_m}.\]
This differs from the quantity
\[d(m+1) - \dim \lkah{\alpha}\]
that one might hope would be a suitable ``log jet codimension.'' There is a correction term. 

\end{remark}

The article \cite{arcspacedifflsdefernexdocampo} goes on to obtain a numerical version of the birational transformation rule. This rule fails in log geometry.

\begin{remark}[The birational transformation rule]

The birational transformation rule \cite[Theorem pg. 18]{devlinmallorymotivicintegrationnotes} does not easily extend to log jet spaces. We defer to future work on ``log motivic integration'' and merely comment on geometric avatars of the theorem here. 

Suppose $f : X \to Y$ is a \emph{strict} morphism of fs log schemes which is an isomorphism away from a closed set $Z \subseteq Y$. Write $W$ for the pullback $Z \times_Y X$. One can prove as in the ordinary case \cite[Theorem pg. 17]{devlinmallorymotivicintegrationnotes} that the map $f$ induces an isomorphism
\[\ljsp{r}{X, \infty} \setminus \ljsp{r}{W, \infty} \longsimeq \ljsp{r}{Y, \infty} \setminus \ljsp{r}{Z, \infty}.\]

In general, $f$ will not induce such an isomorphism even if the map is birational. The map $\Aff^1 \to (\Aff^1)^\circ$ is the identity on underlying schemes, but the log jet space of the former has infinitely many more components. 

\end{remark}

\begin{remark}
The computation of the embedding dimension can be done for other Weil restrictions. When $S'/S$ consist of curves or germs of curves, an order of vanishing expression \eqref{eqn:ordfittingjacobian} holds as well. 
\end{remark}

\section{The case $r = 0$}\label{s:r=0}

When $r = 0$, the chart $\NN \to k \adj{t}_m$ sends $1 \in \NN$ to $t^0 = 1$. Since the copy of $\NN$ maps to a unit, it is erased upon taking the associated log structure. The log structure of the log jets for $r = 0$ is the trivial one
\[\Gamma(M_{\ljet{0}{m}}) = (k \adj{t}_m)^*.\]
That is, $\ljet{0}{m} = (\jet{m})^\circ$. 

This trivial log structure is initial, so there are maps
\[\ljet{r}{m} \to \ljet{0}{m}\]
inducing maps on jet spaces
\[\ljsp{0}{X, m} \to \ljsp{r}{X, m}.\]
The $r = 0$ jet space is initial among all possible ``log jet spaces.'' 

The jet space $r = 0$ is the only one that maps directly to $X$ because $\ljsp{0}{X, 0} = X$. The truncation maps are compatible with the maps between different jet spaces. 

\begin{example}

If $r = 0$ and $m = 1$, the relative log jet space $\ljsp{0}{X/S, m}$ is the log tangent space $\Tl{X/S}$ by Remark \ref{rmk:logtanspfunctorofpoints}. This is in contrast with the case $r > 0$ developed in this article. 

\end{example}

This space is the log jet space of S. Dutter \cite{sethdutterthesislogjets} and B. Fleming \cite{balinthesis}. Fleming defines a scheme without log structure, taking as input ordinary $X$-schemes $T$ and endowing them with the final log structure $M_T = \OO_T$. The jets $\jet{m}$ are endowed with the initial log structure, and the functor of points is defined as usual
\[\ljsp{0}{X, m}(T) \coloneqq \Hom((T, \OO_T) \times \jet{m}^\circ, X).\]

\begin{remark}

By extending this functor of points to arbitrary fs log schemes $T$, we get $\ljsp{0}{X, m}$. Restricting to final log structures $M_T = \OO_T$ on $X$-schemes $T$ is equivalent to endowing the non-fs version $\ljspnonfs{0}{X, m}$ of $\ljsp{0}{X, m}$ with log structure pulled back from $X$ \cite[Remark III.5]{balinthesis}. Both approaches make the map on log structures vacuous. It is thus more accurate to say their jet space is the underlying scheme $\left(\ljspnonfs{0}{X, m}\right)^\circ$ with log structure pulled back from $X$. 

\end{remark}

Suppose $T = \Spec A$ for simplicity. The log structure on $T \times \ljet{0}{m}$ is the pushout
\[M_{T \times \ljet{0}{m}} = A \oplus_{A^*} (A \adj{t}_m)^*.\]
Fleming points out that this monoid may be described as $A \times (1 + t A \adj{t}_m)$ by pushing the units onto the left factor.

If $X \to Y$ is log \'etale, there is a simpler pullback square for $r = 0$
\[\begin{tikzcd}
\ljsp{0}{X, m+1} \ar[r] \ar[d] \lpb      &\ljsp{0}{Y, m+1} \ar[d]         \\
\ljsp{0}{X, m} \ar[r] \ar[d] \lpb      &\ljsp{0}{Y, m} \ar[d]         \\
X \ar[r]       &Y.
\end{tikzcd}\]

\begin{example}[Log jets of $\Aff^1$]

As before, the truncation maps $\ljsp{0}{\Aff^1, m+1} \to \ljsp{0}{\Aff^1, m}$ are $\left(\Aff^1\right)^\circ$-bundles. The base case $m=0$ is the original space $\ljsp{0}{\Aff^1, 0} = \Aff^1$, so the log jets for $r = 0$ are
\[\ljsp{0}{\Aff^1, m} = \left(\Aff^m\right)^\circ \times \Aff^1.\]

\end{example}

\begin{remark}

Compared to the arc space, less is known about the wedge space $\jsp{\jsp{X, \infty}, \infty} = \HHom(\jet{\infty} \times \jet{\infty}, X)$. For example, is there a birational transformation rule? B. Fleming asks \cite[Remark III.42]{balinthesis} whether log jets will help understand what happens to these for log \'etale maps. 

Let $X \to Y$ be log \'etale. Applying Lemma \ref{lem:evspmaplogetmonom} twice, we see that
\[\ev \ev X \to \ev \ev Y\]
is log \'etale. The pullback square
\[\begin{tikzcd}
\ljsp{r}{\ljsp{r}{X, \infty}, \infty} \ar[r, "f_{\infty, \infty}"] \ar[d] \lpb       &\ljsp{r}{\ljsp{r}{X, \infty}, \infty} \ar[d]      \\
\ljsp{r}{\ev X} \ar[r] \ar[d] \lpb      &\ljsp{r}{\ev Y} \ar[d]         \\
\ev \ev X \ar[r]       &\ev \ev Y
\end{tikzcd}\]
shows the induced map on wedge spaces is log \'etale also. Note the double evaluation space is the evaluation space for $\NN^2$ mentioned in Remark \ref{rmk:Qevsp} $\ev \ev - = \ev_{\NN^2} -$. 

In Fleming's case $r = 0$, any other properties of $X \to Y$ stable under pullback also hold for the map $f_{\infty, \infty}$ on wedge spaces. Log modifications lead to log modifications of wedge spaces. This birational transformation rule for log wedge spaces is in the spirit of Fleming's question. 

\end{remark}

A major difference between the log jets is how those with $r = 0$ restrict along components. Let $X$ be an fs log scheme and write $X_{j}$ for the locally closed stratum where the log structure of $X$ has rank $j$. As shown in \cite[Lemma 3.2]{karustaallogjets} (and \cite[Proposition III.12]{balinthesis} for positive characteristic), the maps 
\[\ljsp{0}{X, m} \to \jsp{X^\circ, m}\]
restrict over the strata $X_j \subseteq X$ to affine bundles of varying dimension $mj$. 

The reason is that the completions of the diagram 
\[
\begin{tikzcd}
A \times (1 + t A \adj{t}_m) \ar[d]        &Q \ar[l, dashed] \ar[d]     \\
A \adj{t}_m         &R \ar[l]
\end{tikzcd}
\]
can be identified with group homomorphisms $\gp Q \to (1 + t A \adj{t}_m)$. Here $Q \to R$ is a chart for the log structure of $X$ and the left hand side is an $A$-valued log jet. The log structures $M_{\Spec A \times \ljet{r}{m}}$ for $r \neq 0$ are not given by $A \times (1 + t A \adj{t}_m)$, so this restriction formula does not occur.

We conclude the section with analogues of our main results, proved the same way. Note that $\ev X$ must be replaced with $X$ itself in the proofs.

\begin{theorem}

Let $X$ be an fs log scheme. 

\begin{itemize}
    \item
The log jets $\ljsp{0}{X, m}$ with $r = 0$ are representable by a log algebraic space. If $X$ admits charts Zariski locally, they are representable by a log scheme. 

    \item If $X$ is log smooth, the truncation maps $\ljsp{0}{X, m+1} \to \ljsp{0}{X, m}$ are affine bundles. 

    \item Bhatt's theorem holds for $r = 0$
\[\ljsp{0}{X, \infty} = \lim_m \ljsp{0}{X, m}\]

    \item The log tangent space of the log jets with $r=0$ is again the Weil restriction
    \[\Tl{\ljsp{0}{X/S, m}/S} = \WR_{\cal U/\ljsp{0}{X/S, m}} \Tl{X/S}|_{\cal U}\]
    
    The formula for the log K\"ahler differentials holds as well
    \[\lkah{\ljsp{0}{X, m}/S} = \rho_*\left(\lkah{X/S}|_\cal U \otimes \omega \right).\]

\end{itemize}

\end{theorem}

\appendix

\section{Jet spaces of stacks}

The representability of hom-stacks \cite[Theorem 3.7(iii)]{rydhrepability} shows the jet spaces of algebraic stacks are algebraic. We check this directly and obtain representability for log jet spaces of log algebraic stacks as a corollary. 

Fix the order of jets $m \in \NN$ for this section; we do not consider arcs here for simplicity. Bhatt gives a stacky example where arcs differ from the limit of jet spaces \cite{bhatttannakaalgn}. 

Jet spaces were defined for Artin $N$-stacks in \cite{chetanbalwejetspacesforartinnstacks}, while  \cite{nlab:higherjetspaces} offers further refinements. We redefine them concretely here due to the authors' inexperience with higher stacks. We do not allow stack structure on the jets $\Delta_m$ themselves as in \cite{yasuda2003motivic}.

\begin{definition}\label{def:jetspacesofstacks}
Let $X$ be a stack over affine schemes over $\Spec k$ with the \'etale topology. The \textit{jet space} of $X$ is the fibered category
\[\jsp{X} : \Spec A \mapsto \Hom(\Spec A \times \Delta_m, X).\]
This fibered category maps to $X$ via precomposing by $\Spec A \times \Delta_m \to \Spec A$. This construction is 2-functorial in $X$
\[f : X \to Y \quad \quad \overset{J}{\mapsto} \quad \quad \jsp{f} : \jsp{X} \to \jsp{Y},\]
and we have commutative squares
\[\begin{tikzcd}
\jsp{X} \ar[r]\ar[d]        &\jsp{Y}  \ar[d]      \\
X \ar[r]       &Y.
\end{tikzcd} \]   
\end{definition}

We will show the 2-functor $J$ in Definition \ref{def:jetspacesofstacks} sends Artin stacks to Artin stacks. This section is an exercise in the definitions. 

\begin{remark}\label{rmk:jetspacesareastack}
An \'etale cover of affines $\{\Spec B_i \to \Spec A\}$ gives rise to an \'etale cover $\{\Spec B_i \adj{t}_m \to \Spec A \adj{t}_m\}$ by pullback, so the stack axioms of $\jsp{X}$ reduce to those of $X$. That is, $J_-$ sends stacks to stacks. The same argument would work in any other topology. 

This same argument would not work for $m = \infty$ even for the Zariski topology, since
\[A \stquot{\dfrac{1}{f}} \adj{t} \neq A \adj{t} \stquot{\dfrac{1}{f}}.\]
Elements like $\sum \dfrac{1}{f^i} t^i$ belong to the former but not the latter. 

\end{remark}

\begin{lemma}\label{lem:difflcriteriajetspacespresentation}
Suppose $f : X \to Y$ is a morphism of stacks over schemes in the \'etale topology which is representable and locally finite presentation. If $f$ is smooth, then $\jsp{f} : \jsp{X} \to \jsp{Y}$ is smooth and $\jsp{X} \to \jsp{Y} \times_Y X$ is a surjection. If $f$ is a smooth surjection, so is $\jsp{f}$. 

If instead $f$ is unramified, then $\jsp{X} \to \jsp{Y} \times_Y X$ is a monomorphism. If $f$ is \'etale, $\jsp{X} \to \jsp{Y} \times_Y X$ is an isomorphism. 
\end{lemma}

\begin{proof}

Suppose given a squarezero extension $A' \to A$ of algebras endowed with either of these equivalent commutative squares
\[\begin{tikzcd}
\Spec A \ar[r]  \ar[d]    &\jsp{X} \ar[d]        &      &\Spec A\adj{t}_m \ar[r] \ar[d]       &X \ar[d]      \\
\Spec A' \ar[r] \ar[ur, dashed]        &\jsp{Y}     &       &\Spec A'\adj{t}_m \ar[r] \ar[ur, dashed]        &Y.
\end{tikzcd}\]
Then the dashed arrow on the right exists by smoothness, which verifies the lifting property for $X \to Y$ \cite[02GZ]{sta}. 

If $f$ is instead a smooth surjection, the map $\jsp{f}$ factors as $\jsp{X} \to \jsp{Y} \times_Y X \to \jsp{Y}$. A similar lifting diagram argument verifies $\jsp{X} \to \jsp{Y} \times_Y X$ is surjective, and surjectivity of $\jsp{Y} \times_Y X \to \jsp{Y}$ is immediate. 

The statement about unramified $f$ comes from a similar diagram 
\[\begin{tikzcd}
\Spec A \ar[r] \ar[d]         &X \ar[d]      \\
\Spec A \adj{t}_m \ar[r] \ar[ur, dashed, shift left] \ar[ur, shift right, dashed]       &Y.
\end{tikzcd}\]
The case of \'etale $f$ results from the other cases \cite[025G]{sta}.

\end{proof}

If $f$ is merely surjective, $\jsp{f}$ may not be. Any nilpotent immersion should serve as a counterexample, such as $f : \Spec \CC \in \Spec \CC[x]/x^2$. The lemma holds in greater generality, as it appeals only to formal criteria.

\begin{corollary}
The map $\jsp{X} \to X$ is representable by algebraic spaces if $X$ is a DM stack. 
\end{corollary}

\begin{lemma}\label{lem:repabilityofdiagonalforjetspaces}
Let $X$ be a stack over affine schemes with the \'etale topology. If the diagonal of $X$ is representable by algebraic spaces or by schemes, the same is true of the diagonal of $\jsp{X}$. 
\end{lemma}

\begin{proof}

We will pull back the diagonal $\Delta_{\jsp{X}}$ along a map from an affine $\Spec A$ and determine its functor of points for an affine $\Spec B$
\[\begin{tikzcd}
\Spec B \ar[dr] \ar[rr, bend left] \ar[dashed, r]         &\jsp{X} \times_{\jsp{X} \times \jsp{X}} \Spec A \ar[r] \ar[d] \pb       &\jsp{X} \ar[d]        \\
        &\Spec A \ar[r]        &\jsp{X} \times \jsp{X}.
\end{tikzcd}\]
Unpacking, this becomes
\[\begin{tikzcd}
\Spec B\adj{t}_m \ar[dr] \ar[rr, bend left] \ar[r, dashed]        &X \times_{X \times X} \Spec A \adj{t}_m \ar[r] \ar[d] \pb        &X \ar[d]      \\
        &\Spec A \adj{t}_m \ar[r]      &X \times X.
\end{tikzcd}\]
This gives us an identification 
\[\jsp{X} \times_{\jsp{X} \times \jsp{X}} \Spec A = \jsp{X \times_{X \times X} \Spec A \adj{t}_m}.\]
We've shown the diagonal of $\jsp{X}$ is representable by \textit{the jet space of} an algebraic space or scheme under the corresponding assumption on $X$. The jet space of a scheme is a scheme, but it remains to show the same for algebraic spaces. 

\textbf{Claim:} The jet space of an algebraic space is an algebraic space.

The diagonal of an algebraic space is representable by schemes, so the same is true of its jet space by the above argument. Given an \'etale surjection from a scheme onto the algebraic space, Lemma \ref{lem:difflcriteriajetspacespresentation} transforms this into an \'etale surjection on jet spaces. 

\end{proof}

\begin{corollary}
The jet space of an Artin stack is an Artin stack. The jet space of a DM stack is DM. 
\end{corollary}

\begin{proof}

Lemma \ref{lem:repabilityofdiagonalforjetspaces} gives us representability of the diagonal and Lemma \ref{lem:difflcriteriajetspacespresentation} gives us a smooth or \'etale surjection from a scheme onto the jet space.

\end{proof}

\begin{corollary}\label{cor:ljetrepablestacks}
The log jet space $\ljsp{r}{X, m}$ of a log algebraic stack $X$ is a log algebraic stack. 
\end{corollary}

\begin{proof}

The pullback square \ref{rmk:logstackpbsquare} exhibits $\ljsp{r}{X, m}$ as a fiber product of log algebraic stacks. 

\end{proof}

\begin{remark}
The construction $X \mapsto \jsp{X}$ also respects groupoid presentations. The 2-functor of taking jet spaces respects limits as well as smooth morphisms as in Lemma \ref{lem:difflcriteriajetspacespresentation}.

\end{remark}

\bibliographystyle{unsrt}%Used BibTeX style is unsrt
\bibliography{zbib}
\end{document}